\documentclass[11pt,a4paper]{article}
\usepackage[utf8]{inputenc}
\usepackage[english]{babel}
\usepackage{amsmath,amssymb,mathrsfs,amsthm}
\usepackage{graphicx}
\usepackage{anysize}
\usepackage{verbatim}
\usepackage{enumerate}
\usepackage{courier}
\usepackage{epsfig}
\newcommand{\CC}{\mathbb C}
\newcommand{\RR}{\mathbb R}

\newcommand{\ep}{\varepsilon}

\newcommand{\TT}{\mathbb T}
\newcommand{\pat}{\partial_t}
\newcommand{\pax}{\partial_x}

\newcommand{\paa}{\partial_\alpha}
\newcommand{\da}{\partial_\alpha}

\newcommand{\al}{\alpha}

\newcounter{comentcount}
\newcounter{teocount}
\newcounter{propcount}
\allowdisplaybreaks[2]

\newtheorem{lem}{Lemma}

\newtheorem{coro}{Corollary}
\newtheorem{teo}[teocount]{Theorem}
\newtheorem{defi}{Definition}

\newenvironment{coment}
{\stepcounter{comentcount} {\bf \tt Remark} {\bf\tt\arabic{comentcount}} }{ }
\title{On turning waves for the inhomogeneous Muskat problem: a computer-assisted proof}
\author{Javier G\'omez-Serrano$^{\mbox{{\footnotesize 1}}}$ and Rafael Granero-Belinch\'on$^{\mbox{{\footnotesize 2}}}$}
\begin{document}

\maketitle
\footnotetext[1]{Email: \texttt{jg27@math.princeton.edu},\\
Princeton University,\\
Department of Mathematics, \\
Fine Hall, Washington Road, \\
Princeton, NJ 08544-1000.
}
\footnotetext[2]{Email: \texttt{rgranero@math.ucdavis.edu},\\
Department of Mathematics,\\
University of California, Davis,\\
One Shields Avenue,\\
Davis, 95616.
}

\vspace{0.3cm}
\begin{abstract}
We exhibit a family of graphs that develop turning singularities (i.e. their Lipschitz seminorm blows up and they cease to be a graph, passing from the stable to the unstable regime) for the inhomogeneous, two-phase Muskat problem where the permeability is given by a nonnegative step function. We study the influence of different choices of the permeability and different boundary conditions (both at infinity and considering finite/infinite depth) in the development or prevention of singularities for short time. In the general case (inhomogeneous, confined) we prove a bifurcation diagram concerning the appearance or not of singularities when the depth of the medium and the permeabilities change. The proofs are carried out using a combination of classical analysis techniques and computer-assisted verification.
\end{abstract}
\vspace{0.3cm}



\textbf{Acknowledgments}: The authors are supported by the Grant MTM2011-26696 from Ministerio de Ciencia e Innovaci\'on (MICINN) and MINECO: ICMAT Severo Ochoa project SEV-2011-0087. Javier G\'omez-Serrano is supported by StG-203138CDSIF of the ERC. Rafael Granero-Belinch\'on is grateful to Luigi Berselli and Rafael Orive for productive comments in an early version of these results. Javier G\'omez-Serrano thanks Rafael de la Llave for fruitful discussions. We thank Diego C\'ordoba for his guidance and useful suggestions. We wish to thank the Instituto de Ciencias Matem\'aticas (Madrid) for computing facilities.

\maketitle

\section{Introduction}\label{IIIsec1}
In this paper we study the evolution of the interface between two different incompressible fluids with the same viscosity in a two-dimensional porous medium. This problem is worthwhile studying since it is a model of an aquifer or an oil well (see \cite{muskat1937flow} and the references therein) or a model of a geothermal reservoir (see \cite{CF} and the references therein). We address the differences between the dynamics of the singularity of turning waves when the assumptions of the model change. In this context, we will refer to a turning singularity whenever we speak about curves such that initially have a point with vertical tangent, backwards in time can be parametrized as graphs and forward in time they can not, as seen in Figure \ref{FigTurning}.

\begin{figure}[h!]\centering
\includegraphics[scale=0.35]{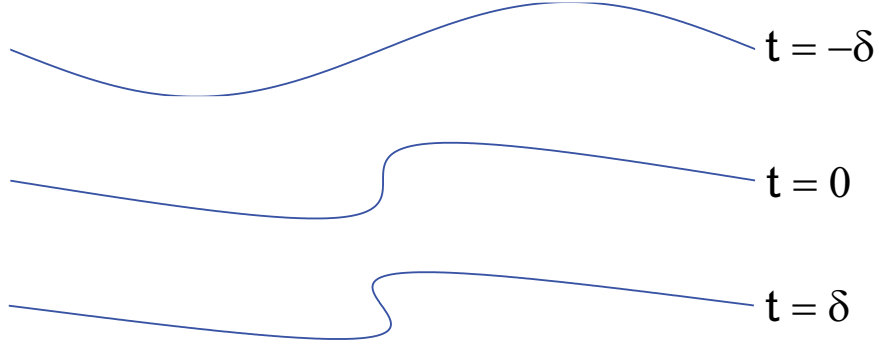}
\caption{Turning singularity: graph, vertical tangent and turning of the interface.}
\label{FigTurning}
\end{figure}

We notice that, according to our definition, a point with vertical tangent, by itself, is not a turning singularity, since the curve can recoil and move into the stable regime where it can be parametrized as a graph. In order to be considered as a singularity, the curve necessarily has to turn over. Parametrized as a curve, the interface remains analytic while parametrized as a graph has a singularity. In other words, the parametrization blows up. In this framework, the singularity is equivalent to the fact that $\partial_\alpha z_1<0$ at some point for short time.

We consider two incompressible fluids with the same viscosity but different densities, $\rho^1$ and $\rho^2$, evolving in a two dimensional porous medium with permeability $\kappa(x)$. The velocity field obeys Darcy's law:
\begin{equation}\label{IIIdarcy}
\mu\frac{v}{\kappa}=-\nabla p-g(0,\rho)^t,
\end{equation}
where $\mu$ is the viscosity and $g$ is the acceleration due to gravity, and the incompressibility condition
\begin{equation}\label{IIIincom}
\nabla\cdot v=0.
\end{equation}
We take $\mu=g=1$. The fluids also satisfy the conservation of mass equation
\begin{equation}\label{IIIconser}
\pat\rho+v\cdot\nabla\rho=0.
\end{equation}
Given $l>0$, the spatial domains considered are $\Omega=\RR\times(-l,l),\RR^2$ and $\TT\times\RR.$ We denote by $S^1$ the volume occupied by the fluid with density $\rho^1$ and by $S^2$ the volume occupied by the fluid with density $\rho^2$. The interface between both fluids is the curve $z(\alpha,t)$.
Given $0<h_2<l$, we consider that the permeability is
\begin{equation}\label{IIIperm}
\kappa(x)=\kappa^1\textbf{1}_{\{(x,y)\in\Omega, y>-h_2\}}+\kappa^2\textbf{1}_{\{(x,y)\in\Omega, y\leq-h_2\}},
\end{equation}
\emph{i.e.} the curve $h(\alpha)=(\alpha,-h_2)$ separates the regions with different permeabilities. We assume that the initial curve $z(\alpha,0)$ does not touch the curve $h(\alpha)$. Moreover we consider that $z(\alpha,0)$ is in the region with permeability equal to $\kappa^1$. See Figure \ref{FigInhomogeneous} for an illustration of the previous domains.

\begin{figure}[h!]\centering
\includegraphics[scale=0.35]{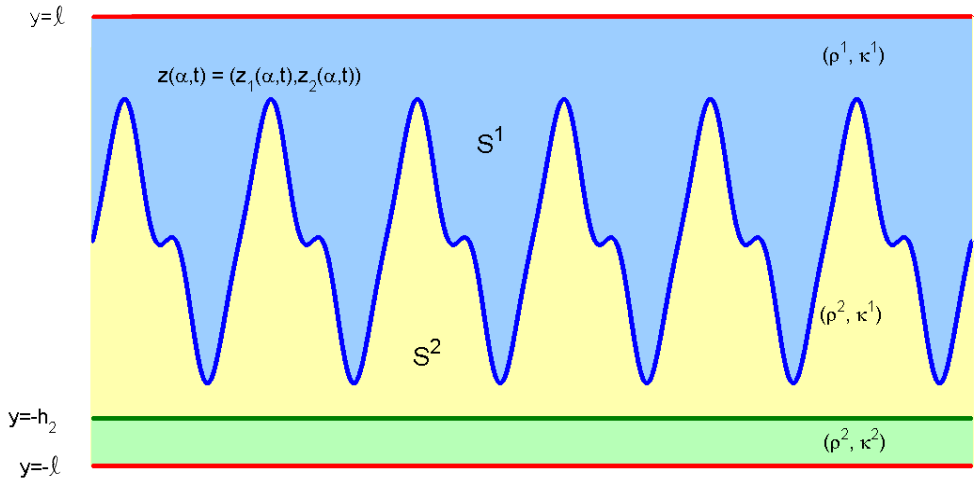}
\caption{Situation of the different fluids and permeabilities.}
\label{FigInhomogeneous}
\end{figure}

We define the Rayleigh-Taylor condition
$$
RT(\alpha,t)=-(\nabla p^2(z(\alpha,t))-\nabla p^1(z(\alpha,t)))\cdot\partial_\alpha^\bot z(\alpha,t).
$$
Fix $t>0$. If $RT(\alpha,t)>0 , \;\forall\alpha\in\RR$ we will say that the curve is in the Rayleigh-Taylor stable regime and if $RT(\alpha,t)<0$ for some $\alpha$, we will say that the curve is in the Rayleigh-Taylor unstable regime. We note that when the interface is a graph and is parametrized as $(\alpha,f(\alpha))$, this function reduces to
$$
RT=g(\rho^2-\rho^1),
$$
and the curve is in the RT stable regime whenever $\rho^1<\rho^2$.

The Muskat problem where the permeability is constant and the depth is infinite has been studied in many works. A proof of local existence of classical solutions in the Rayleigh-Taylor stable regime and ill-posedness in the unstable regime can be encountered in \cite{c-g07}. %\todo{Rafa: seguro? ojo porque he cambiado la frase}%
A maximum principle for $\|f(t)\|_{L^\infty}$ can be found in \cite{c-g09}. Moreover, the authors showed in \cite{c-g09} that if $\|\pax f_0\|_{L^\infty}<1$, then $\|\pax f(t)\|_{L^\infty}<\|\pax f_0\|_{L^\infty}$. In \cite{ccfgl}, the authors determined that the initial curve becomes analytic for every positive time and they also proved the existence of turning singularities. For other results see \cite{ambrose2004well, castro2012breakdown, Castro-Cordoba-Gancedo:recent-results-muskat, ccgs-10, c-g-o08, e-m10, KK, SCH}.

The case with finite depth (equivalently, when the permeability is supported in the strip $\RR\times(-l,l)$) has been addressed in \cite{Cordoba-GraneroBelinchon-Orive:confined-muskat}. In this work the authors found the existence and uniqueness of solutions in the RT stable regime, a smoothing effect, ill-posedness in the RT unstable regime, a maximum principle and a decay estimate for $\|f(t)\|_{L^\infty}$ which is slower than in the case where the depth is infinity. The authors also proved that if the initial datum has small amplitude and slope (in a very precise sense depending on the depth), $\|\pax f(t)\|_{L^\infty}$ verifies a uniform bound and under more restrictive conditions for the initial amplitude and slope, the derivative obeys a maximum principle. We remark that the condition is not only on the size of the slope. Moreover, in this region there are global weak solutions (see \cite{G}).

The Muskat problem where the permeability is given by \eqref{IIIperm} has been treated in \cite{Berselli-Cordoba-GraneroBelinchon:local-solvability-singularities-muskat}. For this model the authors proved well-posedness and the existence of turning singularities when the physical parameters are in a precise range. One of our main contributions is to extend the range of physical parameters where the waves turn by means of a computer-assisted proof. In \cite{Berselli-Cordoba-GraneroBelinchon:local-solvability-singularities-muskat} there is numerical evidence that supports the aforementioned results.

In the present paper we compare different models. First, we show the existence of waves such that when the depth is $l=\pi/2$ the wave turns and if $l=\infty$ then the slope of the wave decreases for a short enough time. In \cite{Cordoba-GraneroBelinchon-Orive:confined-muskat} there is numerical evidence of this result. The same result is true when the Muskat equation is replaced with the water waves equations (or free boundary incompressible Euler equations, see \cite{Castro-Cordoba-Fefferman-Gancedo-GomezSerrano:finite-time-singularities-water-waves-surface-tension, Castro-Cordoba-Fefferman-Gancedo-GomezSerrano:splash-water-waves, Castro-Cordoba-Fefferman-Gancedo-GomezSerrano:finite-time-singularities-free-boundary-euler} and the references therein), which are given by
\begin{equation}\label{ww}
\left\{\begin{array}{lll}
\pat z(\alpha,t) & = & BR(\varpi,z)z(\alpha,t)+c(\al,t)\paa z(\al,t) \\
\pat \varpi(\alpha,t) & = & -2\pat BR(\varpi,z)z(\alpha,t)\cdot\paa z(\alpha,t)-\paa\left(\frac{|\varpi|^2}{4|\paa z|}\right)+\paa\left(c\varpi\right)\\
&&+2c(\al,t)\paa BR(\varpi,z)z(\alpha,t)\cdot\paa z(\al,t)-2g\paa z_2(\al,t),
\end{array}\right.
\end{equation}
where $z$ is the interface, $\varpi$ is the amplitude of the vorticity, $c$ accounts for the reparametrization freedom of the curve and $BR$ denotes the Birkhoff-Rott kernel (see \eqref{IIIeq2} in Section \ref{IIIsec2} below). Notice that this kernel depends on the domain.

Second, we study a model where the permeability is given by a nonnegative step function. In this case, we are interested in the effect of this inhomogeneity in the interface. For this model we obtain that with different permeabilities there is no global in time solution in the Rayleigh-Taylor stable regime corresponding to an arbitrary, large (in $C^1$) initial data which is a graph. Moreover, if the permeabilities verify some conditions we get that they can help or prevent the formation of turning singularities for some families of initial data. These results are true for both the periodic and the flat at infinity cases.

Finally, we consider the most general model where, in addition to the change of permeabilities, the medium is bounded by impervious walls. For this case, we define a family of curves depending on the height $h_2$ where the permeability jump is located in a way that the curves are located above $h_2$. For this family, we perform rigorous computations of a bifurcation diagram in which the parameters are $h_2$ and the permeability values and the outcomes are $\{\text{turning, not turning, unknown}\}$. We obtain that the family exhibits different behaviours depending on $h_2$: for some of them the outcome is independent of the permeability values and for some it is not. Moreover, we see that the property of turning/not turning is persistent, \emph{i.e.} small variation of the parameters give rise to the same outcome and we prove the existence of a smooth curve in parameter space that delimits turning from not turning.
%at least for 97.1\% of the values of the parameters, the property of turning/not turning is persistent, \emph{i.e.} small variation of the parameters give rise to the same outcome.

The role of the permeability is rather subtle. Assuming that the initial data can be parametrized as a graph, in \cite[Section 4]{Berselli-Cordoba-GraneroBelinchon:local-solvability-singularities-muskat}, the numerics show that if $\kappa^1-\kappa^2<0$ the evolution for $\|f(t)\|_{L^\infty}$ is smoother than in the case with only one permeability ($\kappa^1-\kappa^2=0$) in the sense that the decay of this quantity is faster. In the same way, if $\kappa^1-\kappa^2>0$ the decay of the $L^\infty$ norm is slower than in the homogeneous case. However, when the evolution of $\|\pax f(t)\|_{L^\infty}$ is addressed, the same numerics show that the situation is reversed. For the Lipschitz seminorm, the decay is faster in the case $\kappa^1-\kappa^2>0$ than in the homogeneous case. 

The use of computers to perform floating-point arithmetic can lead to numerical errors. To overcome this difficulty and prove rigorous results, we use the so-called \emph{interval arithmetics}, in which instead of working with arbitrary real numbers, we perform computations over intervals which have representable numbers as endpoints. On these objects, an arithmetic is defined in such a way that we are guaranteed that for every $x \in X, y \in Y$
\begin{align*}
x \star y \in X \star Y,
\end{align*}

for any operation $\star$. For example,
\begin{align}
[\underline{x},\overline{x}] + [\underline{y},\overline{y}] & = [\underline{x} + \underline{y}, \overline{x} + \overline{y}] \nonumber \\
[\underline{x},\overline{x}] \times [\underline{y},\overline{y}] & = [\min\{\underline{x}\underline{y},\underline{x}\overline{y},\overline{x}\underline{y},\overline{x}\overline{y}\},\max\{\underline{x}\underline{y},\underline{x}\overline{y},\overline{x}\underline{y},\overline{x}\overline{y}\}] \nonumber \\
\max\{[\underline{x},\overline{x}],[\underline{y},\overline{y}]\} & = [\max\{\underline{x},\underline{y}\},\max\{\overline{x},\overline{y}\}].
\label{defMAX}
\end{align}
We can also define the interval version of a function $f(X)$ as an interval $I$ such that for every $x \in X$, $f(x) \in I$. Rigorous computation of integrals has been theoretically developed since the seminal work of Moore and many others \cite{Berz-Makino:high-dimensional-quadrature,Kramer-Wedner:adaptive-gauss-legendre-verified-computation,Lang:multidimensional-verified-gaussian-quadrature,Moore-Bierbaum:methods-applications-interval-analysis}, and has had applications in physics \cite{Holzmann-Lang-Schutt:gravitation-verified-quadrature}. An important ingredient of our proofs will be the rigorous computation of some integrals. Having a tight enclosure of the result is crucial for the sake of determining if an initial condition will develop a turning singularity or not for short time. In order to perform the rigorous computations we used the C-XSC library \cite{CXSC}.

The organization of this paper is as follows: the contour equations are obtained in Section \ref{IIIsec2}, a precise statement of the theorems is given in Section \ref{MainResultsMuskat}, their proofs can be found in Section \ref{SectionProofs} and the codes in the supplementary material. The codes are intended to be read in order. Some of the strategies of the Theorems are built upon the ones used for the previous ones. Moreover, we have sacrificed performance for readability in the first 3 Theorems where the computation was less intensive than in the last one. In any case, we have tried to achieve the optimal asymptotic complexity but without optimizing in a very deep low level.

\textbf{Notation:} we denote $(a,b)^\perp=(-b,a)$ and define
$$
\mathcal{K}=\frac{\kappa^1-\kappa^2}{\kappa^1+\kappa^2}\text{ and } \bar{\rho}=\frac{\kappa^1(\rho^2-\rho^1)}{4\pi}.
$$
We notice that $\mathcal{K}$ is an dimensionless number and satisfies $-1 < \mathcal{K} < 1$. From now on, we also drop the dependence in $t$.

For readability purposes, instead of writing the intervals as for example $[123456,123789]$ we will refer to them as $123^{456}_{789}$.

\section{The contour equation}\label{IIIsec2}
In this section we obtain the contour equation. Now we consider the bounded porous medium $\RR\times (-l,l)$. This regime is equivalent to the case with more than two $\kappa^i$ because the boundaries can be understood as regions with $\kappa=0$. Given a scalar function $a$ and curves $f=(f_1,f_2),g=(g_1,g_2)$, we denote the Birkhoff-Rott integral by
\begin{equation}
\label{IIIeq2}
BR(a,g)f(\alpha)=\text{P.V.}\int_\RR a(\beta) BS(f(\alpha),g(\beta))d\beta,
\end{equation}
where $BS$ denotes the Biot-Savart law in $\RR\times(-l,l)$, which is given by the kernel (see \cite{Cordoba-GraneroBelinchon-Orive:confined-muskat})
\begin{multline}\label{IIIBSconf}
BS(x,y,\mu,\nu)=\frac{1}{8l}\left(\frac{-\sin\left(\frac{\pi}{2l}(y-\nu)\right)}{\cosh\left(\frac{\pi}{2l}(x-\mu)\right)-\cos\left(\frac{\pi}{2l}(y-\nu)\right)}+\frac{\sin\left(\frac{\pi}{2l}(y+\nu)\right)}{\cosh\left(\frac{\pi}{2l}(x-\mu)\right)+\cos\left(\frac{\pi}{2l}(y+\nu)\right)},\right.\\
\left.\frac{\sinh\left(\frac{\pi}{2l}(x-\mu)\right)}{\cosh\left(\frac{\pi}{2l}(x-\mu)\right)-\cos\left(\frac{\pi}{2l}(y-\nu)\right)}-\frac{\sinh\left(\frac{\pi}{2l}(x-\mu)\right)}{\cosh\left(\frac{\pi}{2l}(x-\mu)\right)+\cos\left(\frac{\pi}{2l}(y+\nu)\right)}\right).
\end{multline}
To simplify notation we take the depth to be $l=\pi/2$. Notice that if $z(\al)$ is a solution of the Muskat problem \eqref{IIIdarcy}-\eqref{IIIconser} with depth $l=\pi/2$, then $z^\lambda(\al,t)=\lambda z(\lambda\al,t/\lambda)$ is the interface corresponding to a solution of the Muskat problem with depth equal to $l=\pi\lambda/2$. 

%Indeed, to simplify the exposition, let's assume that the evolution is given only by the first term in the second component of \eqref{IIIBSconf}. Then we have
%\begin{multline*}
%\pat z^\lambda(\al,t)=\pat z(\lambda\al,t/\lambda)\\
%=\frac{1}{4\pi}\text{P.V.}\int_\RR\frac{(\pax z(\lambda \al,t/\lambda)-\pax z(y,t/\lambda))\sinh(z_1(\lambda \al,t/\lambda)-z_1(y,t/\lambda))}{\cosh(z_1(\lambda\al,t/\lambda)-z_1(y,t/\lambda))-\cos(z_2(\lambda \al,t/\lambda)-z_2(y,t/\lambda))}dy.
%\end{multline*}
%By changing variables, we get
%$$
%\pat z^\lambda(\al,t)=\frac{1}{4\pi\lambda}\text{P.V.}\int_\RR\frac{(\pax z^\lambda(\al,t)-\pax z^\lambda(y,t))\sinh\left(\frac{z^{\lambda}_1(\al,t)-z^{\lambda}_1(y,t)}{\lambda}\right)}{\cosh\left(\frac{z^{\lambda}_1(\al,t)-z^{\lambda}_1(y,t)}{\lambda}\right)-\cos\left(\frac{z^{\lambda}_2(\al,t)-z^{\lambda}_2(y,t)}{\lambda}\right)}dy.
%$$

Due to \eqref{IIIdarcy},\eqref{IIIincom}, \eqref{IIIconser} and \eqref{IIIperm} the vorticity concentrates on the two interfaces as long as a weak solution exists for the full system considered. Thus, we can write it as
\begin{equation}\label{IIIvort}
\omega(\alpha,t)=\varpi_1(\alpha,t)\delta((x,y)-z(\alpha,t)) + \varpi_2(\alpha,t)\delta((x,y)-h(\alpha)),
\end{equation}
where $\varpi_1$ and $\varpi_2$ stand for the different vorticity amplitudes. Computing the limits of the velocity towards the two interfaces we see that
\begin{equation}
\label{IIIeq3}
v^\pm(z(\alpha))=\lim_{\epsilon\rightarrow0}v(z(\alpha)\pm\epsilon\paa^\perp z(\alpha))=BR(\varpi_1,z)z(\al)+BR(\varpi_2,h)z(\al)\mp\frac{1}{2}\frac{\varpi_1(\alpha)}{|\paa z(\alpha)|^2}\paa z(\alpha),
\end{equation}
and
\begin{equation}
\label{IIIeq4}
v^\pm(h(\alpha))=\lim_{\epsilon\rightarrow0}v(h(\alpha)\pm\epsilon\paa^\perp h(\alpha))=BR(\varpi_1,z)h(\al)+BR(\varpi_2,h)h(\al)\mp\frac{1}{2}\frac{\varpi_2(\alpha)}{|\paa h(\alpha)|^2}\paa h(\alpha).
\end{equation}
We observe that $v^+(z(\alpha))$ is the limit inside $S^1$ (the upper subdomain) and $v^-(z(\alpha))$ is the limit inside $S^2$ (the lower subdomain). The curve $z(\alpha)$ does not touch the curve $h(\alpha)$, therefore the limits for the curve $h$ are in the same subdomain $S^2$.

Using Darcy's Law, we have
\begin{align*}
(v^-(z(\alpha))-v^+(z(\alpha)))\cdot\paa z(\alpha)&=\kappa^1\left(-\paa(p^-(z(\alpha))-p^+(z(\alpha)))\right)-\kappa^1(\rho^2-\rho^1)\paa z_1(\alpha)\\
&=0-\kappa^1(\rho^2-\rho^1)\paa z_2(\alpha),
\end{align*}
since the pressure is continuous along the interface (see \cite[Section 2]{c-c-g10}). Using \eqref{IIIeq3} we conclude
\begin{equation}
\label{IIIeq5}
\varpi_1(\alpha)=-\kappa^1(\rho^2-\rho^1)\paa z_2(\alpha).
\end{equation}
We need to determine $\varpi_2$. We consider
\begin{eqnarray*}
\left[\frac{v}{\kappa}\right]&\equiv &\left(\frac{v^-(h(\alpha))}{\kappa^2}-\frac{v^+(h(\alpha))}{\kappa^1}\right)\cdot\paa h(\alpha)\\
&=&-\paa (p^-(h(\alpha))-p^+(h(\alpha)))\\
&=&0,
\end{eqnarray*}
where the first equality is obtained by Darcy's Law. Expression \eqref{IIIeq4} leads us to
$$
\left[\frac{v}{\kappa}\right]=\left(\frac{1}{\kappa^2}-\frac{1}{\kappa^1}\right)\left(BR(\varpi_1,z)h(\al)+BR(\varpi_2,h)h(\al)\right)\cdot\paa h(\alpha)+\left(\frac{1}{2\kappa^2}+\frac{1}{2\kappa^1}\right)\varpi_2.
$$
We have a Fredholm integral equation of the second kind:
\begin{equation}\label{IIIw2def}
\varpi_2(\alpha)+\frac{\mathcal{K}}{2\pi}\;\text{P.V.}\int_\RR\frac{\varpi_2(\beta)\sin(2h_2)}{\cosh(\alpha-\beta)+\cos(2h_2)}d\beta=-2\mathcal{K}BR(\varpi_1,z)h(\al)\cdot(1,0).
\end{equation}

We define the Fourier transform as
$$
\mathcal{F}(f)(\zeta)=\frac{1}{\sqrt{2\pi}}\int_{\RR}e^{-ix\zeta}f(x)dx,
$$
and using some of its basic properties, we obtain
$$
\mathcal{F}(\varpi_2)(\zeta)\left(1+\frac{\mathcal{K}}{\sqrt{2\pi}}\mathcal{F}\left(\frac{\sin(2h_2)}{\cosh(x)+\cos(2h_2)}\right)(\zeta)\right)=-2\mathcal{K}\mathcal{F}(BR(\varpi_1,z)h\cdot(1,0))(\zeta).
$$
In \cite{Berselli-Cordoba-GraneroBelinchon:local-solvability-singularities-muskat} the equation for $\varpi_2$ is solved for every $|\mathcal{K}|<\delta(h_2)$ with
\begin{equation}\label{IIIdelta}
\delta(h_2)=\min\left\{1,\frac{\sqrt{2\pi}}{\displaystyle \max_{\zeta}\left|\mathcal{F}\left(\frac{\sin(2h_2)}{\cosh(x)+\cos(2h_2)}\right)(\zeta)\right|}\right\}.
\end{equation}
We have the following result concerning the range of correct parameters:
\begin{lem}\label{lemacitado}
Let $0<h_2<\pi/2$ be a constant, then $\delta(h_2)=1$. Thus, there exists a solution to \eqref{IIIw2def} for every $-1<\mathcal{K}<1$.
\end{lem}
\begin{proof}
We prove the result by computing explicitly
$$
J = \mathcal{F}\left(\frac{\sin(2h_2)}{\cosh(x)+\cos(2h_2)}\right)(\zeta)=\frac{1}{\sqrt{2\pi}}\int_\RR e^{-ix\zeta}\frac{\sin(2h_2)}{\cosh(x)+\cos(2h_2)}dx.
$$
Take $\zeta\in\RR$, $\zeta<0$. We consider the complex extension
$$
I^{j}=\int_{\partial\Gamma^{j}} e^{-iz\zeta}\frac{\sin(2h_2)}{\cosh(z)+\cos(2h_2)}dz,
$$
where $\Gamma^{j}=(-\pi-2j\pi,\pi+2j\pi)\times(0,2j \pi)\in\CC, \; j \in \mathbb{N}$. The poles of the function (see Figure \ref{FigPoles}) are
$$
\gamma_k^-=(\pi-2h_2+2k\pi)i\text{ and }\gamma_k^+=(\pi+2h_2+2k\pi)i, \quad k \in \mathbb{Z}.
$$
Given that
$$
\cosh(z)+\cos(2h_2)=2\cosh\left((z+2h_2i)/2\right)\cosh\left((z-2h_2i)/2\right),
$$
$\gamma_k^\pm$ are simple poles. %Because of the form of $\gamma_k^\pm$ we take $R=2j\pi, j \in \mathbb{N}$.
We split the contour integral in
$$
I^{j}=I_{1}^{j} + I_{2}^{j} + I_{3}^{j} + I_{4}^{j},
$$
with
\begin{align*}
I_{1}^{j}& =\int^{\pi+2j\pi}_{-(\pi+2j\pi)}\frac{e^{-ix\zeta}\sin(2h_2)}{\cosh(x)+\cos(2h_2)}dx, \\
I_{2}^{j}& =\int_{\pi+2j\pi}^{-(\pi+2j\pi)}\frac{e^{-i(x+2 \pi j i)\zeta}\sin(2h_2)}{\cosh(x+2 \pi j i)+\cos(2h_2)}dx, \\
I_{3}^{j}& =\int_0^{2j\pi}\frac{e^{-i(\pi+2j\pi+iy)\zeta}\sin(2h_2)}{\cosh(\pi+2j\pi+iy)+\cos(2h_2)}dy, \\
I_{4}^{j}& =\int_{2j\pi}^0\frac{e^{-i(-\pi-2j\pi+iy)\zeta}\sin(2h_2)}{\cosh(-\pi-2j\pi+iy)+\cos(2h_2)}dy.
\end{align*}

\begin{figure}[h!]\centering
\includegraphics[scale=0.5]{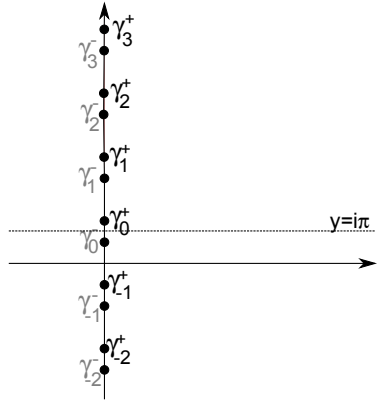}
\caption{Situation of the poles: $\gamma_{i}^{+}$ in black, $\gamma_{i}^{-}$ in grey.}
\label{FigPoles}
\end{figure}

Using classical trigonometric identities, we get
$$
|I_{2}^{j}|\leq\int_{-\infty}^{\infty}\frac{e^{2j\pi\zeta}\sin(2h_2)}{\cosh(x)+\cos(2h_2)}dx\leq c_{h_2}e^{2j\pi\zeta},
$$
which tends to zero when $j$ tends to infinity since $\zeta < 0$. We can bound the third integral as
\begin{multline*}
|I_{3}^{j}|\leq\int_0^{2j\pi}\frac{e^{y\zeta}}{|\cosh(\pi+2j\pi)\cos(y)+\sinh(\pi+2j\pi)\sin(y)+\cos(2h_2)|}dy\\
\leq\int_0^{\infty}\frac{e^{y\zeta}}{(\cosh(\pi+2j\pi)-1)^2+\cos(2h_2)-2}dy.
\end{multline*}
The same remains valid for $I_{4}^{j}$. Then, taking the limit $j\rightarrow\infty$:
$$
J = \lim_{j \to \infty}\frac{1}{\sqrt{2\pi}}I^{j}=\frac{1}{\sqrt{2\pi}}\int^{\infty}_{-\infty}\frac{e^{-ix\zeta}\sin(2h_2)}{\cosh(x)+\cos(2h_2)}dx.
$$
By the Residue Theorem, this implies
\begin{align*}
I^{j}& =2\pi i\sum_{k=0}^{j}\text{Res}\left(\frac{e^{-iz\zeta}\sin(2h_2)}{\cosh(z)+\cos(2h_2)},\gamma_k^\pm\right) \\
J & = \lim_{j \to \infty}\frac{1}{\sqrt{2\pi}}I^{j} = 2\pi i \sum_{k\geq0}\frac{1}{\sqrt{2\pi}}ie^{\pi\zeta}(e^{2\pi\zeta})^k2\sinh(2h_2\zeta)=
\sqrt{2\pi}\frac{\sinh(2h_2\zeta)}{\sinh(\pi\zeta)}.
\end{align*}
The result can be easily extended to every $\zeta > 0$ by the evenness of $\frac{\sin(2h_2)}{\cosh(x)+\cos(2h_2)}$ and to $\zeta = 0$ by the the continuity of the Fourier transform. Finally we obtain
$$
\delta(h_2)=\min\left\{1,\frac{\pi}{2h_2}\right\},
$$
and we conclude that $\delta(h_2)=1$ for every $0<h_2<\pi/2$. Moreover, this extends the result in \cite[Remark 2]{Berselli-Cordoba-GraneroBelinchon:local-solvability-singularities-muskat}, where numerical evidence of its validity was found.
\end{proof}
Thus, for every $|\mathcal{K}|<1$, we can write the expression of $\varpi_2$ as
\begin{align}
\varpi_2(\al)=&-2\mathcal{K}BR(\varpi_1,z)h(\al)\cdot(1,0)+\frac{2\mathcal{K}^2}{2\pi}BR(\varpi_1,z)h(\al)\cdot(1,0)*G_{h_2,\mathcal{K}}\nonumber\\
=&\, 2\mathcal{K}\bar{\rho}\left[\text{P.V.}\int_\RR\paa z_2(\beta)\frac{\sin(h_2+z_2(\beta))}{\cosh(\alpha-z_1(\beta))-\cos(h_2+z_2(\beta))}d\beta\nonumber\right.\\
&-\text{P.V.}\int_\RR\paa z_2(\beta)\frac{\sin(-h_2+z_2(\beta))}{\cosh(\alpha-z_1(\beta))+\cos(-h_2+z_2(\beta))}d\beta\nonumber\\
&-\frac{\mathcal{K}}{2\pi}G_{h_2,\mathcal{K}}*\text{P.V.}\int_\RR\frac{\paa z_2(\beta)\sin(h_2+z_2(\beta))}{\cosh(\alpha-z_1(\beta))-\cos(h_2+z_2(\beta))}d\beta\nonumber\\
&+\left.\frac{\mathcal{K}}{2\pi}G_{h_2,\mathcal{K}}*\text{P.V.}\int_\RR\frac{\paa z_2(\beta)\sin(-h_2+z_2(\beta))}{\cosh(\alpha-z_1(\beta))+\cos(-h_2+z_2(\beta))}d\beta\right], \label{IIIw2defb}
\end{align}
where
$$
G_{h_2,\mathcal{K}}(\xi)=\mathcal{F}^{-1}\left(\frac{\mathcal{F}\left(\frac{\sin(2h_2)}{\cosh(x)+\cos(2h_2)}\right)}{1+\frac{\mathcal{K}}{\sqrt{2\pi}}\mathcal{F}\left(\frac{\sin(2h_2)}{\cosh(x)+\cos(2h_2)}\right)}\right)=\int_\RR\frac{\cos(y\xi)\sinh(2h_2 y)}{\sinh(\pi y)+\mathcal{K}\sinh(2h_2 y)}dy.$$

We observe that $G_{h_2,\mathcal{K}}$ is a function in the Schwartz class. Using
$$
\int_\RR\partial_\beta\log\left(\cosh(\al-z_1(\beta))\pm\cos(y\pm z_2(\beta))\right)d\beta=0,
$$
and adding the correct tangential term (see \cite{Berselli-Cordoba-GraneroBelinchon:local-solvability-singularities-muskat, c-c-g10, Cordoba-GraneroBelinchon-Orive:confined-muskat}), we obtain
\begin{align}
\pat z(\alpha)=&\bar{\rho}\text{P.V.}\int_\RR\frac{(\paa z(\alpha)-\paa z(\beta))\sinh(z_1(\alpha)-z_1(\beta))}{\cosh(z_1(\alpha)-z_1(\beta))-\cos(z_2(\alpha)-z_2(\beta))}d\beta\nonumber\\
&+\bar{\rho}\text{P.V.}\int_\RR\frac{(\paa z_1(\alpha)-\paa z_1(\beta),\paa z_2(\alpha)+\paa z_2(\beta))\sinh(z_1(\alpha)-z_1(\beta))}{\cosh(z_1(\alpha)-z_1(\beta))+\cos(z_2(\alpha)+z_2(\beta))}d\beta\nonumber\\
&+\frac{1}{4\pi}\text{P.V.}\int_\RR\varpi_2(\beta)BS(z_1(\alpha),z_2(\alpha),\beta,-h_2)d\beta\nonumber\\
&+\frac{\paa z(\alpha)}{4\pi}\text{P.V.}\int_\RR\varpi_2(\beta)\frac{\sin(z_2(\alpha)+h_2)}{\cosh(z_1(\alpha)-\beta)-\cos(z_2(\alpha)+h_2)}d\beta\nonumber\\
&+\frac{\paa z(\alpha)}{4\pi}\text{P.V.}\int_\RR\varpi_2(\beta)\frac{\sin(z_2(\alpha)-h_2)}{\cosh(z_1(\alpha)-\beta)+\cos(z_2(\alpha)-h_2)}d\beta\label{IIIeqv2}.
\end{align}
In the case where the fluids fill the whole plane, we can take the limit $l\rightarrow\infty$ in \eqref{IIIBSconf} and write
\begin{multline}
\label{IIIeq9}
\pat z(\alpha)=2\bar{\rho}\text{P.V.}\int_\RR\frac{z_1(\alpha)-z_1(\beta)}{|z(\alpha)-z(\beta)|^2}(\paa z(\alpha)-\paa z(\beta))d\beta\\
+\frac{1}{2\pi}\text{P.V.}\int_\RR\varpi_2(\beta)\frac{(z(\alpha)-h(\beta))^\perp}{|z(\alpha)-h(\beta)|^2}d\beta\\
+\paa z(\alpha)\frac{1}{2\pi}\text{P.V.}\int_\RR\varpi_2(\beta)\frac{z_2(\alpha)+h_2}{|z(\alpha)-h(\beta)|^2}d\beta,
\end{multline}
with
\begin{equation}
\varpi_2(\alpha)=4\mathcal{K}\bar{\rho}\text{P.V.}\int_\RR\paa z_2(\beta)\frac{h_2+z_2(\beta)}{|h(\alpha)-z(\beta)|^2}d\beta,\label{IIIeq10A}
\end{equation}
and, if the initial curve is periodic in the horizontal variable, using complex variables notation for the curve $z=(z_1,z_2)=z_1+iz_2$ and the identity
$$
\frac{1}{z}+\sum_{k\geq1}\frac{2z}{z^2-(2k\pi)^2}=\frac{1}{2\tan(z/2)},\;\;\forall z\in\CC,
$$
we get
\begin{multline}\label{IIIeq13}
\pat z(\alpha)=\bar{\rho}\text{P.V.}\int_\TT\frac{\sin(z_1(\alpha)-z_1(\beta))(\paa z(\alpha)-\paa z(\beta))d\beta}{\cosh(z_2(\alpha)-z_2(\beta))-\cos(z_1(\alpha)-z_1(\beta))}\\
+\frac{\paa z_1(\alpha)-1}{4\pi}\text{P.V.}\int_\TT\frac{\sinh(z_2(\alpha)+h_2)\varpi_2(\beta)d\beta}{\cosh(z_2(\alpha)+h_2)-\cos(z_1(\alpha)-\beta)}\\
+\frac{i}{4\pi} \text{P.V.}\int_\TT\frac{(\paa z_2(\alpha)\sinh(z_2(\alpha)+h_2)+\sin(z_1(\alpha)-h_1(\beta)))\varpi_2(\beta)d\beta}{\cosh(z_2(\alpha)+h_2)-\cos(z_1(\alpha)-\beta)},
\end{multline}
where the second vorticity amplitude can be written as
\begin{equation}
\varpi_2(\alpha)=2\bar{\rho}\mathcal{K} \text{P.V.}\int_\TT\frac{\sinh(h_2+z_2(\beta))\paa z_2(\beta)d\beta}{\cosh(-h_2-z_2(\beta))-\cos(\alpha-z_1(\beta))}.\label{IIIeq14A}
\end{equation}
\section{Statement of the results}
\label{MainResultsMuskat}
In this section we will state the theorems that will be proved in the next one. We show that the fact of having a confined medium plays a role in the mechanism for achieving turning singularities. Moreover, we also show that there are cases for which the jump in the permeabilities can lead to either prevent or promote these singularities, and cases in which the heterogeneity of the medium has no impact on whether the wave turns or not.

Notice that the confined (and homogeneous) Muskat problem corresponds to $\varpi_2=0$ in \eqref{IIIeqv2}, while the unconfined (and homogeneous) satisfies $\varpi_2=0$ in \eqref{IIIeq9}. For these cases we have the next theorem:

\begin{teo}
There exists a family of analytic curves $z(\al) = (z_1(\al),z_2(\al))$, flat at infinity, for which there exists a finite time $T$ such that the solution to the confined Muskat problem develops a turning singularity before $t=T$ and the non confined does not.
\label{ThmConfTurnsNoConfNoTurns}
\end{teo}

This theorem also implies the following result:
\begin{coro}\label{water}
There exists a family of analytic curves $z(\al) = (z_1(\al),z_2(\al))$, flat at infinity, for which there exists a finite time $T$ such that the solution to the confined water waves problem develops a turning singularity before $t=T$ and the non confined does not.
\end{coro}

\begin{coment}
The shallowness parameter (see \cite{Bona-Lannes-Saut:asymptotic-internal-waves}) is defined as
$$
\mu\equiv \left(\frac{\text{\emph{typical} depth}}{\text{\emph{ typical } wavelength}}\right)^2.
$$
Waves with $\mu \ll 1$ are in the shallow water regime. In this example we have $\mu=\frac{1}{16}$, thus the shallow water regime is reached.
\end{coment}

\begin{defi}

We will say that an analytic initial condition $z(\al) = (z_1(\al),z_2(\al))$ \textbf{\emph{turns unconditionally}} if there exists a finite time $T > 0$ for which the solution to the Muskat problem with initial condition equal to $z(\al)$ develops a turning singularity before time $T$ independently of the permeability parameter $\mathcal{K}$. Analogously, we will say that it \textbf{\emph{recoils unconditionally}} if there exists a finite time $T > 0$ for which the solution to the Muskat problem with initial condition equal to $z(\al)$ does not develop a turning singularity before time $T$ independently of the permeability parameter $\mathcal{K}$. If the curve $z(\al,t) = (z_1(\al),z_2(\al))$ does not satisfy any of the two conditions mentioned before, we will say that the initial condition \textbf{\emph{turns conditionally}}.

\end{defi}

\begin{coment}
These behaviours are \emph{local in time}, thus, they refer to times $0\leq t \leq T$. In other words, for some initial data that \emph{turns unconditionally}, there may exist some $T_2>T$ such that for $t>T_2$ the interface can be parametrized as a graph. The converse is also true: there may exist initial data such that they become smooth graphs for $0<t<T$ and $T_2>T$ such that 
$$
\limsup_{t\rightarrow T_2}\|\pax f(t)\|_{L^\infty}=\infty.
$$
\end{coment}

\begin{defi}We will say that for a given analytic initial condition $z(\al) = (z_1(\al),z_2(\al))$, the permeabilities \textbf{\emph{help}} the formation of singularities if the curve turns conditionally and for $\mathcal{K} = 0$ there exists a time $T > 0$ such that it does not develop a turning singularity before time $T$. Analogously, we will say that they \textbf{\emph{prevent}} the formation of singularities if the curve turns conditionally and for $\mathcal{K} = 0$ there exists a time $T > 0$ such that it develops a turning singularity before time $T$.
\end{defi}

For the unconfined, inhomogeneous Muskat problem we have
\begin{teo}
There exist 3 different families of analytic curves $z(\al) = (z_1(\al),z_2(\al))$, periodic in the horizontal variable such that the corresponding solution to the unconfined, inhomogeneous Muskat \eqref{IIIeq13}:
\begin{itemize}
\item[(a)] They turn unconditionally.
\item[(b)] The permeabilities help the formation of singularities.
\item[(c)] The permeabilities prevent the formation of singularities.
\end{itemize}
\label{ThmDepKappaIndepKappaPeriodic}
\end{teo}

\begin{teo}
There exist 3 different families of analytic curves $z(\al) = (z_1(\al),z_2(\al))$, flat at infinity such that the corresponding solution to the unconfined, inhomogeneous Muskat \eqref{IIIeq9}:
\begin{itemize}
\item[(a)] They turn unconditionally.
\item[(b)] The permeabilities help the formation of singularities.
\item[(c)] The permeabilities prevent the formation of singularities.
\end{itemize}
\label{ThmDepKappaIndepKappaFlat}
\end{teo}
Moreover, for both Theorems, in cases $(b)$ and $(c)$, there exists a unique parameter $K^{*}$ such that for all $\mathcal{K} < K^{*}$ the curve exhibits one behaviour and for all $\mathcal{K} > K^{*}$ it exhibits the other.

\begin{coment}
We should remark that Theorems \ref{ThmDepKappaIndepKappaPeriodic} and \ref{ThmDepKappaIndepKappaFlat} are more general than the ones in \cite[Theorems 3 and 4]{Berselli-Cordoba-GraneroBelinchon:local-solvability-singularities-muskat} since we are suppressing any smallness assumption in $|\mathcal{K}|$ or largeness in $h_2$.
\end{coment}

Finally, regarding the confined, inhomogeneous problem, we prove:

\begin{teo}
\label{ThmConfinedInhomogeneous}
There exists a family of analytic initial data $z(\al,h_2) = (z_1(\al,h_2),z_2(\al,h_2))$, depending on the height at which the permeability jump is located, such that the corresponding solution to the confined, inhomogeneous Muskat \eqref{IIIeqv2}:
\begin{enumerate}[(a)]
\item 
\begin{enumerate}[1.]
\item For all $0.25 < h_2 < h_2^{ntu} = 0.648$, the curve recoils unconditionally.
\item For all $ 0.676 < h_2 < 0.686$, the permeabilities help the formation of singularities.
\item For all $ 0.715 < h_2 <  0.738$, the permeabilities prevent the formation of singularities.
\item For all $ 0.77 = h_2^{tu} < h_2 < 1.25$, the curve turns unconditionally.
\end{enumerate}
\item There exists a $C^1$ curve $(h_2,\mathcal{K}(h_2))$, located in $[0.648,0.77] \times (-1,1)$, such that for every $h_2$ for which the curve is defined, for every $\mathcal{K}<\mathcal{K}(h_2)$ the curve does not turn and for every $\mathcal{K}>\mathcal{K}(h_2)$ the curve turns.
\end{enumerate}
\end{teo}

\section{Proof of the Theorems}
\label{SectionProofs}
The idea of these proofs is to transform the problem on the turning or not into finding a sign of a given quantity ($\da v_1(0,0)$). This sign will be validated using interval arithmetics. First, we consider curves such that $\paa z_1(0,0)=0$ and define $\displaystyle m(t)=\min_\alpha \paa z_1(\alpha,t)$. We will assume that $m(0)=\paa z_1(0,0)=0$ holds, and this minimum is only attained at $\al = 0$. Now, if $\paa v_1(0,0)=\partial_\alpha\pat z_1(0,0)>0$ then we get $\frac{d}{dt}m(t)>0$ for $t>0$ small enough. This implies $m(\delta)>0$ for a small enough $\delta>0$ and the curve can be parametrized as a graph. Indeed, we compute
$$
m(\delta)=m(0)+\int_0^ \delta \frac{d}{dt}m(s)ds=\int_0^ \delta \pat\paa z_1(\alpha_s,s)ds>0,
$$
where $\al_s$ is a point where the minimum is attained. If $\paa v_1(0,0)=\partial_\alpha\pat z_1(0,0)<0$, then $m(t)<0$ if $t$ is small enough and the curve can not be parametrized as a graph. After this goal is achieved, we approximate our initial data with analytic curves with the same properties (for instance by convolving it with the heat kernel). All these analytic curves that approximate our explicit constructed example satisfy the same symmetry hypotheses (see below). For these approximating curves, we apply the local existence forward and backward in time theorems proved in \cite{Berselli-Cordoba-GraneroBelinchon:local-solvability-singularities-muskat, ccfgl, Cordoba-GraneroBelinchon-Orive:confined-muskat}.

%\subsection{Proof of Theorem \ref{ThmConfTurnsNoConfNoTurns}}
\subsection{The homogeneous case}\label{SubsectionProofThm1}
In this section we prove that the boundaries make the Muskat problem more singular from the point of view of singularity formation. Equivalently, the boundaries decrease the diffusion rate (see \cite{Cordoba-GraneroBelinchon-Orive:confined-muskat}).
\begin{proof}[Proof of Theorem \ref{ThmConfTurnsNoConfNoTurns}]
We take $l=\pi/2$, $\bar{\rho} = 1$.%$\kappa^1=1$ and $\rho^2-\rho^1=4\pi$.
We consider curves $z(\alpha)=(z_1(\alpha),z_2(\alpha))$ such that:
\begin{enumerate}
 \item $z_i$ are analytic, odd functions.
 \item $\partial_\alpha z_1(\alpha)>0, \forall \alpha\neq 0$, $\partial_\alpha z_1(0)=0$, and $\partial_\alpha z_2(0)>0$.
\end{enumerate}
We want to show that $\partial_\alpha v_1(0,0)=\partial_\alpha\pat z_1(0,0)<0$. The equation for this regime is
\begin{multline*}
\pat z(\alpha)=\text{P.V.}\int_\RR\frac{(\paa z(\alpha)-\paa z(\alpha-\beta))\sinh(z_1(\alpha)-z_1(\alpha-\beta))}{\cosh(z_1(\alpha)-z_1(\alpha-\beta))-\cos(z_2(\alpha)-z_2(\alpha-\beta))}d\beta\\
+\text{P.V.}\int_\RR\frac{(\paa z_1(\alpha)-\paa z_1(\alpha-\beta),\paa z_2(\alpha)+\paa z_2(\alpha-\beta))\sinh(z_1(\alpha)-z_1(\alpha-\beta))}{\cosh(z_1(\alpha)-z_1(\alpha-\beta))+\cos(z_2(\alpha)+z_2(\alpha-\beta))}d\beta.
\end{multline*}
Taking one derivative we get
$$
\paa \pat z_1(\alpha)=I_1(\alpha)+I_2(\alpha)+I_3(\alpha),
$$
where
$$
I_1(0)=\text{P.V.}\int_\RR\left(\frac{-\paa^2 z_1(-\beta)\sinh(-z_1(-\beta))}{\cosh(-z_1(-\beta))-\cos(-z_2(-\beta))}+\frac{-\paa^2 z_1(-\beta)\sinh(-z_1(-\beta))}{\cosh(-z_1(-\beta))+\cos(z_2(-\beta))}\right)d\beta,
$$
$$
I_2(0)=\text{P.V.}\int_\RR\left(\frac{\cosh(-z_1(-\beta))\left(-\paa z_1(-\beta)\right)^2}{\cosh(-z_1(-\beta))-\cos(-z_2(-\beta))}+\frac{\left(-\paa z_1(-\beta)\right)^2\cosh(-z_1(-\beta))}{\cosh(-z_1(-\beta))+\cos(z_2(-\beta))}\right)d\beta,
$$
and
\begin{multline*}
I_3(0)=-\text{P.V.}\int_\RR\frac{\left[\sinh(-z_1(-\beta))\left(-\paa z_1(-\beta)\right)\right]^2}{\left(\cosh(-z_1(-\beta))-\cos(-z_2(-\beta))\right)^2}d\beta\\
-\text{P.V.}\int_\RR\frac{(-\paa z(-\beta))\sinh(-z_1(-\beta))\left[\sin(-z_2(-\beta))\left(-\paa z_2(-\beta)\right)\right]}{\left(\cosh(-z_1(-\beta))-\cos(-z_2(-\beta))\right)^2}d\beta\\
+\text{P.V.}\int_\RR\frac{\left[-\paa z_1(-\beta)\sinh(-z_1(-\beta))\right]^2}{\left(\cosh(-z_1(-\beta))+\cos(z_2(-\beta))\right)^2}d\beta\\
+\text{P.V.}\int_\RR\frac{-\paa z_1(-\beta)\sinh(-z_1(-\beta))\left[-\sin(z_2(-\beta))\paa z_2(-\beta)\right]}{\left(\cosh(-z_1(-\beta))+\cos(z_2(-\beta))\right)^2}d\beta.
\end{multline*}
Then, after some integration by parts and using the properties of $z_i$, we get the following expression for the derivative of the velocity in the confined case:
\begin{multline*}
I^A_{neg}\equiv\frac{\partial_\alpha v_1(0)}{2}=\partial_{\alpha} z_2(0)\int_0^\infty \partial_\alpha z_1(\eta)\sinh(z_1(\eta))\sin(z_2(\eta))\bigg{(}\frac{1}{(\cosh(z_1(\eta))-\cos(z_2(\eta)))^2}\\
+\frac{1}{(\cosh(z_1(\eta))+\cos(z_2(\eta)))^2}\bigg{)}d\eta.
\end{multline*}
With the same approach, for the unconfined case the expression is
$$
I^A_{pos}\equiv\frac{\partial_\alpha v_1(0)}{8}=\partial_{\alpha} z_2(0)\int_0^\infty \frac{\partial_\alpha z_1(\eta)z_1(\eta)z_2(\eta)}{(z_1(\eta))^2+(z_2(\eta))^2)^2}d\eta.
$$

Thus, we are left to validate the following signs:
\begin{align}
\label{signconfnoconf}
I^A_{neg}<0,\quad I^A_{pos}>0.
\end{align}

\begin{figure}[h!]\centering
\includegraphics[scale=0.35]{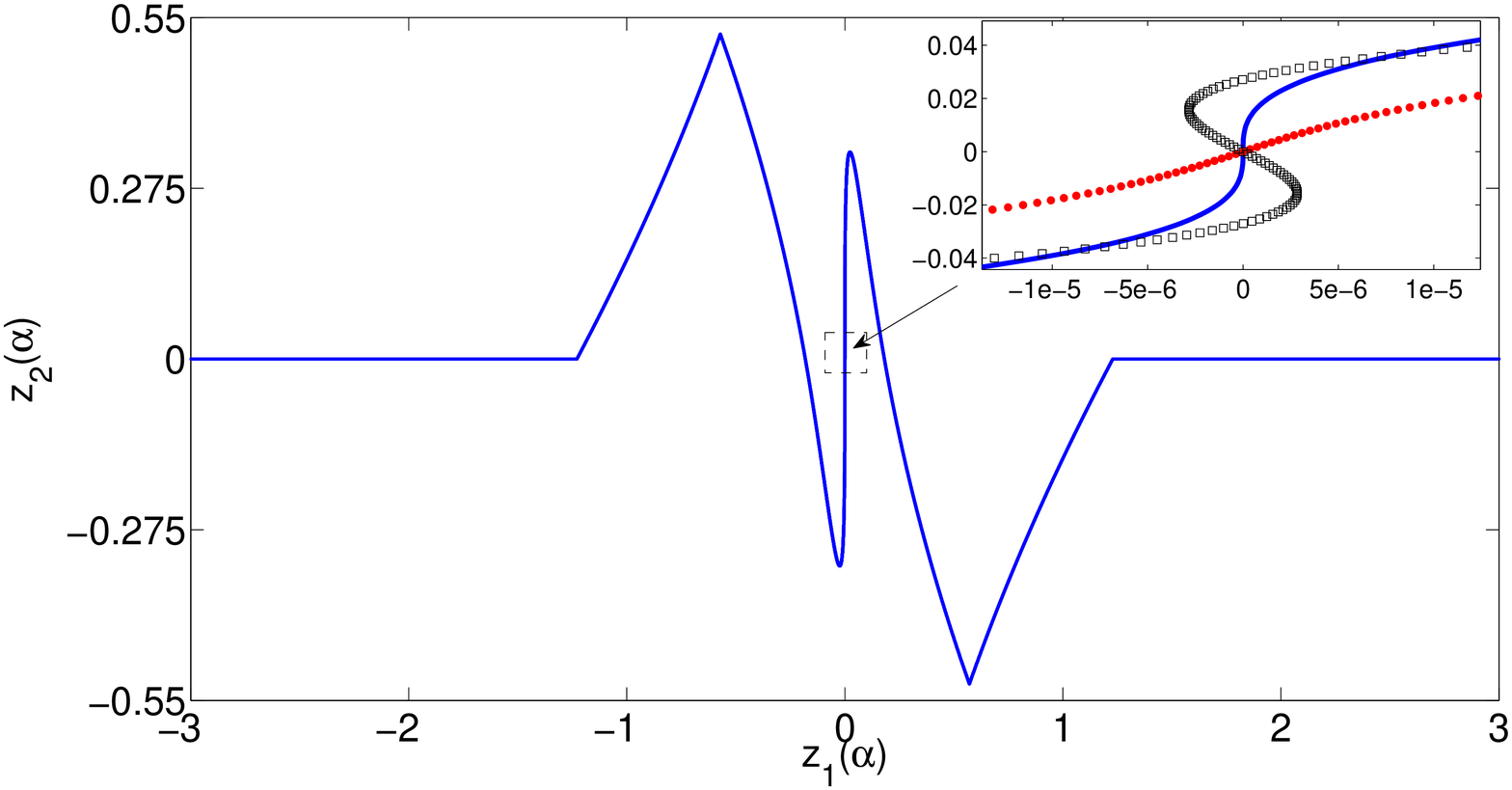}
\caption{The curve in Theorem \ref{ThmConfTurnsNoConfNoTurns}. Inset: Close caption around zero, solid: initial condition, dotted: normal component of the velocity for the infinitely deep case, squared: normal component of the velocity for the finitely deep case. The normal components have been scaled by a factor $1/100$.}
\label{figteo1}
\end{figure}

We rigorously validate them for the following data (see Figure \ref{figteo1}):

\begin{align*}
z_1(\al) & = \al - \sin(\al)e^{-K\al^{2}}, \quad K = 10^{-4} \\
z_2(\al) & = \left\{
\begin{array}{lr}
\displaystyle \frac{\sin(3\al)}{3} & \displaystyle \text{ if } 0 \leq \al \leq \frac{\pi}{3} \vspace{0.2cm} \\
\displaystyle - \al + \frac{\pi}{3} & \displaystyle \text{ if } \frac{\pi}{3} \leq \al \leq \frac{\pi}{2} \vspace{0.2cm} \\
\displaystyle \al - \frac{2\pi}{3} & \displaystyle \text{ if } \frac{\pi}{2} \leq \al \leq \frac{2\pi}{3}\vspace{0.2cm} \\
\displaystyle 0 & \displaystyle \text{ if } \frac{2\pi}{3} \leq \al, \\
\end{array}
\right.
\end{align*}
where $z_2$ is extended such that it is an odd function. This corresponds to the numerical scenario given in \cite{Cordoba-GraneroBelinchon-Orive:confined-muskat}. A first attempt is to compute the normal velocity of the curve in a nonrigorous way using the integral representations in \eqref{IIIeqv2} and \eqref{IIIeq9} and the trapezoidal rule with an equispaced grid for several points around the point with vertical tangent. In Figure \ref{figteo1} (inset), we plot it for the two scenarios (confined and non confined), both scaled by a factor $1/100$. We can observe that the velocity denoted by squares, which corresponds to the confined case, will make the curve develop a turning singularity, where the dotted one (non-confined case) will force the curve to stay in the stable regime.

In order to validate the sign, we split each of $I_{pos}^{A}$ and $I_{neg}^{A}$ into three pieces, each corresponding to a different piece of the piecewise defined $z_2$ in which $z_2(\al)$ is not identically 0.

In the second and third pieces, the integrand is analytic and we can apply Simpson's rule on a uniform (equispaced) mesh $\eta_{0} < \eta_{1} < \ldots < \eta_{N+1}$ for the computation of the integrals:

\begin{align*}
\int_{a}^{b} f(\eta) d\eta \in \sum_{i=0}^{N} \int_{\eta_{i}}^{\eta_{i+1}}f(\eta)d\eta & = \sum_{i=0}^{N} \frac{(\eta_{i+1}-\eta_{i})}{6}\left(f(\eta_{i})+f(\eta_{i+1})+4f\left(\frac{\eta_{i}+\eta_{i+1}}{2}\right)\right)\\
& -\frac{1}{2880}(\eta_{i+1}-\eta_{i})^{5}f^{4}([\eta_{i},\eta_{i+1}]).
\end{align*}

The first piece needs special care since the integrand is of type $\frac{0}{0}$ when $\alpha$ goes to zero. We should remark that the function is integrable: the numerator is $O(\alpha^{6})$ and the denominator is $O(\alpha^{4})$ when expanded both around $\alpha = 0$ in the two problematic cases, namely $I^{A}_{pos}$ and the first summand of $I^{A}_{neg}$. We further split the integral into two pieces, one ranging from 0 to $\ep$ and another from $\ep$ to $\frac{\pi}{3}$. In the validation of the theorem, the choice of the constant $\ep$ equal to $\frac{1}{128}$ was enough. The integrand of the second piece is analytic and is calculated as before, while for the first piece we expand both the numerator and the denominator and cancel out the extra factors $\alpha$. In our case this means (for $I_{pos}^{A}$):

\begin{align*}
\int_{0}^{\ep}\frac{\da z_1(\alpha)\sin(z_1(\alpha))\sinh(z_2(\alpha))}{(\cosh(z_2(\alpha))-\cos(z_1(\alpha)))^2}d\alpha
\equiv \int_{0}^{\ep}\frac{\mathcal{N}(\al)}{\mathcal{D}(\al)}d\alpha \in \int_{0}^{\ep} \frac{\sum_{i=0}^{5}a_{i}\alpha^{i} + \frac{1}{6!}\da^{6} \mathcal{N}([0,\ep])\alpha^{6}}{\sum_{j=0}^{3}b_j\alpha^{j} + \frac{1}{4!}\da^{4}\mathcal{D}([0,\ep])\alpha^{4}}d\alpha.
\end{align*}

Since $a_0, \ldots, a_5, b_0, \ldots, b_3$ are zero, we get

\begin{align}\label{pax4D}
\int_{0}^{\ep}\frac{\mathcal{N}(\al)}{\mathcal{D}(\al)}d\alpha \in \int_{0}^{\ep} \frac{4!}{6!}\frac{\al^2\da^{6} \mathcal{N}([0,\ep])}{ \da^{4}\mathcal{D}([0,\ep])}d\alpha
\subset \frac{\ep^{3}}{3}\frac{1}{30}\frac{\da^{6} \mathcal{N}([0,\ep])}{ \da^{4}\mathcal{D}([0,\ep])}.
\end{align}

The code is flexible so that $N$ can be specified by the user of the program. One can see that for small values of $N$, the intervals in which the value of $I^{A}_{pos}, I^{A}_{neg}$ are enclosed are not small enough such that 0 does not belong to them, needing further precision. However, for $N = 8192$ the grid is fine enough to check conditions \eqref{signconfnoconf}. The calculations for $N = 8192$ can be found in Table \ref{tabla1}.
\begin{table}[h]
    \begin{center}
        \begin{tabular}[c]{|c|c|}
            \hline
            Quantity & Enclosure \\
            \hline
            $I^{A}_{pos}$ & $ 0.0212172^{1922}_{7301}$ \\
            \hline
            $I^{A}_{neg}$ & $-0.0136819^{7345}_{1981}$ \\
            \hline
        \end{tabular}
        \caption{Results of Theorem \ref{ThmConfTurnsNoConfNoTurns}.}
        \label{tabla1}
    \end{center}
\end{table}

% Resultados buenos!!!
% // RESULTS: (N = 8192)
% // Positive integral = [ 0.02121721922547791655544457, 0.02121727300709871227857307]
% // Negative integral = [-0.01368197344584986922810810,-0.01368191981936543906439585]

%\begin{align*}
%I^{A}_{pos} & \in [ 0.02121721922547791655544457, 0.02121738013846577106114034], \\
%I^{A}_{neg} & \in [-0.01368197344584986922810810,-0.01368181286188274552173549].
%\end{align*}

The computation took 2.96 seconds on an Intel i5 processor with 4 GB of RAM. Choosing as initial data a sufficiently close analytic perturbation of $z_2$ finishes the theorem.

\end{proof}

Corollary \ref{water} follows:
\begin{proof}
Take the same curve as before and define the initial amplitude for the vorticity as $\paa z_2(\alpha,0)$. With these initial data we have a solution $(z,\varpi)$ of the water waves problem \eqref{ww} and we obtain the result (see \cite{ccfgl} for more details).
\end{proof}

%\subsection{Proofs of Theorems \ref{ThmDepKappaIndepKappaPeriodic} and \ref{ThmDepKappaIndepKappaFlat}}
\subsection{The inhomogeneous, unconfined case}\label{SubsectionProofsThm23}
In this section we prove the existence of turning waves for a physical parameter region bigger than the one in \cite{Berselli-Cordoba-GraneroBelinchon:local-solvability-singularities-muskat}. In both proofs, we consider curves $z(\alpha)=(z_1(\alpha),z_2(\alpha))$ such that:
\begin{enumerate}
 \item $z_i$ are analytic, odd functions.
 \item $\partial_\alpha z_1(\alpha)>0, \forall \alpha\neq 0$, $\partial_\alpha z_1(0)=0$, and $\partial_\alpha z_2(0)>0$.
 \item $|z_2(\alpha)|<h_2$.
\end{enumerate}
\begin{proof}[Proof of Theorem \ref{ThmDepKappaIndepKappaPeriodic}]
In this case, the question whether the interface turns over or not is reduced to find a negative sign (resp. positive) of
\begin{align}
\da z_2(0)& \left(\int_0^\pi\frac{\da z_1(\beta)\sin(z_1(\beta))\sinh(z_2(\beta))}{(\cosh(z_2(\beta))-\cos(z_1(\beta)))^2}d\beta\right.\nonumber \\
& \left.+\frac{1}{4\pi}\int_0^\pi\frac{(\omega^{B}(\beta)+\omega^{B}(-\beta))(-1+\cosh(h_2)\cos(\beta))}{(\cosh(h_2)
-\cos(\beta))^2}d\beta\right),
\label{velocityperiodic}
\end{align}
where $\omega^{B}$ is
\begin{align}
\label{omegaperiodic}
\omega^{B}(\beta)=\mathcal{K}\int_{-\pi}^{\pi}\frac{\sin(\beta-z_1(\gamma))\da z_1(\gamma)}{\cosh(h_2+z_2(\gamma))-\cos(\beta-z_1(\gamma))}d\gamma
\end{align}
and we assume $\bar{\rho} = \frac{1}{2}$. %$\kappa^1(\rho^2-\rho^1)=2\pi$.
We refer to \cite{Berselli-Cordoba-GraneroBelinchon:local-solvability-singularities-muskat} for the computations leading to these expressions. Plugging \eqref{omegaperiodic} into \eqref{velocityperiodic} we have to compute

\begin{align}
\displaystyle I^{B} & \displaystyle \equiv \da z_2(0)\left(\int_0^\pi\frac{\da z_1(\beta)\sin(z_1(\beta))\sinh(z_2(\beta))}{(\cosh(z_2(\beta))-\cos(z_1(\beta)))^2}d\beta\right. \nonumber \\
& \displaystyle +\frac{\mathcal{K}}{4\pi}\int_{0}^{\pi}\int_{-\pi}^{\pi}\frac{\sin(\beta-z_1(\gamma))\da z_1(\gamma)(-1+\cosh(h_2)\cos(\beta))}{(\cosh(h_2)-\cos(\beta))^2} \nonumber\\
&\times \displaystyle \left.\left(\frac{1}{\cosh(h_2+z_2(\gamma))-\cos(\beta-z_1(\gamma))}+\frac{1}{\cosh(h_2+z_2(\gamma))-\cos(-\beta-z_1(\gamma))}\right)d\beta d\gamma\right) \nonumber \\
& \displaystyle \equiv I^{B}_{1} + I^{B}_{2}.
\label{velocityflatbis}
\end{align}

We remark that the integrand of the 2D integral above is regular (does not even have an indetermination such as the 1D one) since we are assuming that $|z_2(\al)| < h_2$. We calculate $I^{B}_{1}$ as in the first case. However, the choice of a uniform grid in $I^{B}_{2}$ leads to high execution times or low precision. In order to ameliorate the performance of the algorithm, we perform the integration using an adaptive algorithm. We will start with the full domain $[0,\pi] \times [-\pi,\pi]$ and in each iteration we will use a 2D Simpson's rule.

\begin{align*}
\int_{a}^{b}\int_{c}^{d} f(x,y) dx dy & \in \frac{(b-a)(d-c)}{36}\left(16f\left(\frac{a+b}{2},\frac{c+d}{2}\right)\right. \\
& +4\left(f\left(a,\frac{c+d}{2}\right)+f\left(b,\frac{c+d}{2}\right)+f\left(\frac{a+b}{2},c\right)+f\left(\frac{a+b}{2},d\right)\right) \\
& +\left.\frac{}{}\left(f\left(b,c\right)+f\left(b,d\right)+f\left(a,c\right)+f\left(a,d\right)\right)\right) \\
& -\frac{(b-a)(d-c)}{2880} \left((b-a)^4\partial^{4}_xf\left([a,b],[c,d]\right) + (d-c)^4\partial^{4}_{y}f\left([a,b],[c,d]\right)\right).
\end{align*}

If the result meets some tolerance requirements in the form of having absolute or relative (with respect to the volume of the integration region - see Tables \ref{tabla2},\ref{tabla3} for the values used) width smaller than two constants (AbsTol and RelTol) we will save it and add it to the total. Otherwise, we bisect our domain by the midpoint in each of the two directions and call the integrator again with the new 4 subdomains recursively. We also keep track of the number of calls to the integrator and limit the depth of the levels of splitting in order to prevent infinite loops or stack overflows because of too stringent tolerances, but this was not necessary for the parameters specified below.

In order to prove the theorem we will take the following curves defined for $\alpha \in [-\pi,\pi]$ and extended periodically in the horizontal variable.

\begin{align*}
z_1(\al) & = \al - \sin(\al), \\
z_2(\al) & = \left\{
\begin{array}{lr}
\displaystyle \frac{\sin(3\al)}{3}- \sin(\al)\left(e^{-(\al+2)^2}+e^{-(\al-2)^2}\right)& \displaystyle \text{ in case }(a).\\
\displaystyle \frac{\sin(2\al)}{2}- \frac{2}{3}\sin(\al)\left(e^{-(\al+2)^2}+e^{-(\al-2)^2}\right) & \displaystyle \text{ in case } (b).\\
\displaystyle \frac{\sin(2\al)}{1.4}- 0.5\sin(\al)\left(e^{-(\al+2)^2}+e^{-(\al-2)^2}\right) & \displaystyle \text{ in case } (c).\\
\end{array}
\right.
\end{align*}

After running the program with the previous data we get the results summarized in Table \ref{tabla2}. This shows the theorem.

\begin{table}[h]
    \begin{center}
        \begin{tabular}{|c|c|c|c|}
            \hline
            Quantity & (a) & (b) & (c) \\
            \hline
            $I^{B}_{1}$ & $-0.7910^{7003}_{6993}$ & $0.124312^{5192}_{6103}$ & $-0.180519^{6014}_{5579}$\\
            \hline
            $I^{B}_{2}$ & $-0.12^{703437}_{699367}$ & $-0.1414^{5494}_{1422}$ & $-0.2127^{8188}_{1946}$ \\
            \hline
            $I^{B}$ & $-0.918^{1044}_{0636}$ & $-0.0171^{4242}_{0161}$ & $-0.393^{3015}_{2390}$\\
            \hline
            Runtime (sec) & 6.10 & 4.98 & 6.25 \\
            \hline
            Number of calls & 7305 & 5677 & 6405 \\
            \hline
            $N$ & \multicolumn{3}{|c|}{8192}\\
            \hline
            (RelTol,AbsTol) & \multicolumn{3}{|c|}{$(10^{-5},10^{-5})$} \\
            \hline
            $(\mathcal{K}, h_2)$ & \multicolumn{3}{|c|}{$(1,\frac{\pi}{2})$} \\
            \hline
        \end{tabular}
        \caption{Results of Theorem \ref{ThmDepKappaIndepKappaPeriodic}.}
        \label{tabla2}
    \end{center}
\end{table}

Notice that since $I^B$ is linear in $\mathcal{K}$, it will change sign at most once. Together with the values at $\mathcal{K} = \{-1,0,1\}$,  it guarantees existence and uniqueness of $K^{*}$. We should remark that although $\mathcal{K}=\pm 1$ are not physical, they are meaningful by understanding them in the sense of the appropriate limit.

\end{proof}

\begin{proof}[Proof of Theorem \ref{ThmDepKappaIndepKappaFlat}]

Let us assume $\bar{\rho} = \frac{1}{2}$. The turning or not (for a short enough time) for the flat at infinity case can be shown to be equivalent \cite{Berselli-Cordoba-GraneroBelinchon:local-solvability-singularities-muskat} to finding a sign of

\begin{align}
\label{velocityflat}
I^{C} \equiv \da z_2(0)\left(\text{P.V.}\int_0^\infty\frac{4\da z_1(\beta)z_1(\beta)z_2(\beta)}{((z_1(\beta))^2+(z_2(\beta))^2)^2}-\frac{1}{2\pi}\frac{(\omega^{C}(\beta)+\omega^{C}(-\beta))\beta^2}{(\beta^2+h_2^2)^2}d\beta\right),
\end{align}
where $\omega^{C}$ is defined by
\begin{align}
\label{omegaflat}
\omega^{C}(\beta)=2\mathcal{K}\text{P.V.}\int_{-\infty}^{\infty}\frac{(h_2+z_2(\gamma))\da z_2(\gamma)}{(h_2+z_2(\gamma))^2+(\beta-z_1(\gamma))^2}d\gamma.
\end{align}

Plugging \eqref{omegaflat} into \eqref{velocityflat} we have to compute
{
\begin{align}
%\begin{array}{rll}
%\displaystyle I^{C} & \displaystyle \equiv \da z_2(0)\left( 4\text{P.V.}\int_0^\infty\frac{\da z_1(\beta)z_1(\beta)z_2(\beta)}{((z_1(\beta))^2+(z_2(\beta))^2)^2}d\beta\right. & \\
%& \displaystyle -\frac{\mathcal{K}}{\pi}\int_0^\infty\int_{-\infty}^{\infty}\frac{(h_2+z_2(\gamma))\da z_2(\gamma)\beta^2}{(\beta^2+h_2^2)^2}
%&\displaystyle \left(\frac{1}{(h_2+z_2(\gamma))^2+(\beta-z_1(\gamma))^2}\right. \\
%& &\displaystyle \left.\left.+\frac{1}{(h_2+z_2(\gamma))^2+(-\beta-z_1(\gamma))^2}\right)d\beta d\gamma\right) \\
%& \displaystyle \equiv I^{C}_{1} + I^{C}_{2}. &
%\end{array}
\begin{array}{rl}
\displaystyle I^{C} & \displaystyle \equiv \da z_2(0)\left( 4\text{P.V.}\int_0^\infty\frac{\da z_1(\beta)z_1(\beta)z_2(\beta)}{((z_1(\beta))^2+(z_2(\beta))^2)^2}d\beta\right.  -\frac{\mathcal{K}}{\pi}\int_0^\infty\int_{-\infty}^{\infty}\frac{(h_2+z_2(\gamma))\da z_2(\gamma)\beta^2}{(\beta^2+h_2^2)^2} \\
&\displaystyle \times \left(\frac{1}{(h_2+z_2(\gamma))^2+(\beta-z_1(\gamma))^2}\right.
\left.\left.+\frac{1}{(h_2+z_2(\gamma))^2+(-\beta-z_1(\gamma))^2}\right)d\beta d\gamma\right) \\
& \displaystyle \equiv I^{C}_{1} + I^{C}_{2}.
\end{array}
\label{velocityflatbis2}
\end{align}
}%

Again, we compute $I^{C}_{1}$ as in Theorem \ref{ThmConfTurnsNoConfNoTurns}. It is important to notice that we are now integrating $I^{C}_{2}$ in an unbounded region. Even in the case that $z_2$ has compact support and the integral in $\gamma$ is different than zero in a compact set, the integral in $\beta$ cannot be reduced to integrate in a bounded region. Therefore, we split $I^{C}_{2}$ into a bounded part and an unbounded one. We now explain how to deal with the latter since the former is computed as in the previous Theorem.

We want to bound
{
\begin{align*}
I^C_{2,ub} \equiv \displaystyle -\frac{\mathcal{K}}{\pi} \da z_2(0) \int_M^\infty\int_{-\pi}^{\pi}\frac{(h_2+z_2(\gamma))\da z_2(\gamma)\beta^2}{(\beta^2+h_2^2)^2}
&\displaystyle \left(\frac{1}{(h_2+z_2(\gamma))^2+(\beta-z_1(\gamma))^2}\right. \\
 &\displaystyle \left.+\frac{1}{(h_2+z_2(\gamma))^2+(-\beta-z_1(\gamma))^2}\right) d\gamma d\beta
\end{align*}
}%
and we will take the following curves:

\begin{align*}
z_1(\al) & = \al - \sin(\al)e^{-K\al^{2}}, \quad K = 10^{-2}, \\
z_2(\al) & = \left\{
\begin{array}{lr}
\displaystyle \left(\frac{\sin(3\al)}{3}- \sin(\al)\left(e^{-(\al+2)^2}+e^{-(\al-2)^2}\right)\right)1_{\{|\al| \leq \pi\}} & \displaystyle \text{ in case }(a).\\
\displaystyle \left(\frac{\sin(2\al)}{2}- 0.85\sin(\al)\left(e^{-(\al+2)^2}+e^{-(\al-2)^2}\right)\right)1_{\{|\al| \leq \pi\}} & \displaystyle \text{ in case } (b).\\
\displaystyle \left(\frac{\sin(2\al)}{1.8}- 0.7\sin(\al)\left(e^{-(\al+2)^2}+e^{-(\al-2)^2}\right)\right)1_{\{|\al| \leq \pi\}} & \displaystyle \text{ in case } (c).\\
\end{array}
\right.
\end{align*}

We will provide bounds for $I^C_{2,ub}$ in this way:

{
\begin{align}
\label{eqBd2D}
|I^C_{2,ub}| \leq \displaystyle \frac{|\mathcal{K}|}{\pi} |\da z_2(0)| \int_M^\infty \frac{\beta^2}{(\beta^2+h_2^2)^2} d\beta \int_{-\pi}^{\pi}
&\displaystyle \left(\frac{|h_2+z_2(\gamma)||\da z_2(\gamma)|}{(h_2+z_2(\gamma))^2+(\beta-z_1(\gamma))^2}\right. \nonumber \\
 &\displaystyle \left.+\frac{|h_2+z_2(\gamma)||\da z_2(\gamma)|}{(h_2+z_2(\gamma))^2+(-\beta-z_1(\gamma))^2}\right)d\gamma
\end{align}
}%
and let
\begin{align*}
G(\beta) \equiv \frac{|\mathcal{K}|}{\pi} \int_{-\pi}^{\pi}
\left(\frac{|h_2+z_2(\gamma)||\da z_2(\gamma)|}{(h_2+z_2(\gamma))^2+(\beta-z_1(\gamma))^2}\right.
 \left.+\frac{|h_2+z_2(\gamma)||\da z_2(\gamma)|}{(h_2+z_2(\gamma))^2+(-\beta-z_1(\gamma))^2}\right)d\gamma.
\end{align*}

It is easy to check that $G(\beta)$ is monotone in $\beta$ for $\beta$ larger than $\|z_1\|_{L^{\infty}(-\pi,\pi)}$. Indeed,

$$ G(\beta) \leq G(M), \quad \text{ if we take } M = 14 \pi,$$
which is our choice of $M$ for the computer verification. Plugging this relation into \eqref{eqBd2D} we obtain

\begin{align}
|I^C_{2,ub}| & \leq  |\da z_2(0)| G(M) \int_M^\infty \frac{\beta^2}{(\beta^2+h_2^2)^2} d\beta \nonumber \\
& = |\da z_2(0)| G(M) \left(\frac{\pi}{4h_2} - \frac{1}{h_2}\arctan\left(\frac{M}{h_2}\right) + \frac{M}{2(h_2^2+M^2)}\right).
\end{align}

Thus, we are left to compute rigorous bounds for $G$. Let us denote by
\begin{align}
\label{defIG}
IG(\beta,\gamma) = \frac{|\mathcal{K}|}{\pi}
\left(\frac{|h_2+z_2(\gamma)||\da z_2(\gamma)|}{(h_2+z_2(\gamma))^2+(\beta-z_1(\gamma))^2}\right.
 \left.+\frac{|h_2+z_2(\gamma)||\da z_2(\gamma)|}{(h_2+z_2(\gamma))^2+(-\beta-z_1(\gamma))^2}\right)
 \end{align}
the integrand of $G$. That means
\begin{align*}
G(\beta) = \int_{-\pi}^{\pi} IG(\beta,\gamma)d\gamma.
\end{align*}
We perform the following integration scheme:
\begin{align*}
\int_{\gamma_i}^{\gamma_{i+1}} IG(\beta,\gamma)d\gamma =
\left\{
\begin{array}{cc}
IG(\beta,[\gamma_{i},\gamma_{i+1}])(\gamma_{i+1}-\gamma_{i}) & \text{ if } 0 \in IG(\beta,[\gamma_{i},\gamma_{i+1}]) \\
\frac{(\gamma_{i+1}-\gamma_{i})}{6}\left(IG(\beta,\gamma_{i})+IG(\beta,\gamma_{i+1})+4IG\left(\beta,\frac{\gamma_{i}+\gamma_{i+1}}{2}\right)\right) & \\
 -\frac{1}{2880}(\gamma_{i+1}-\gamma_{i})^{5}\partial^{4}_{\gamma}IG(\beta,[\gamma_{i},\gamma_{i+1}]) & \text{ otherwise}
\end{array}
\right.
\end{align*}
in which we apply a Simpson rule for the case where the integrand is smooth, otherwise we take the full interval that results in evaluating the integrand in the whole integration interval. We perform the integration in $\gamma$ over a uniform mesh $-\pi = \gamma_0 < \gamma_1 < \ldots < \gamma_{N_2} = \pi,$ $\gamma_i = -\pi + \frac{2\pi}{N_2}i$.

Therefore, adding all the contributions
\begin{align*}
G(M) = \sum_{i=0}^{N_2-1} \int_{\gamma_{i}}^{\gamma_{i+1}}IG(M,\gamma)d\gamma,
\end{align*}
we get the desired bound on $T$. The variable $N_2$ is user-specified in our program. The results are summarized in Table \ref{tabla3}. These prove the Theorem.

\begin{table}[h]
    \begin{center}
        \begin{tabular}{|c|c|c|c|}
            \hline
            Quantity & (a) & (b) & (c) \\
            \hline
            $I^{C}_{1}$ & $-0.745640^{1337}_{0299}$ & $0.00147^{1972}_{2074}$ & $-0.0087191^{9854}_{1782}$\\
            \hline
            $|I^C_{2,ub}|$ & $0.0000^{0000}_{2668}$ & $0.0000^{0000}_{2697}$ & $0.0000^{0000}_{3183}$\\
            \hline
            $I^{C}_{2} - I^C_{2,ub}$ & $-0.020^{6841}_{3465}$ & $-0.011^{6887}_{3785}$ & $-0.009^{9556}_{5855}$\\
            \hline
            $I^{C}$ & $-0.76^{63509}_{59599}$ & $-0.0^{10244}_{09879}$ & $-0.018^{7067}_{2728}$\\
            \hline
            Runtime (sec) & 6.96 & 8.30 & 8.11 \\
            \hline
            Number of calls & 9205 & 9177 & 8805 \\
            \hline
            $(N,N_2)$ & \multicolumn{3}{|c|}{(8192,256)}\\
            \hline
            (RelTol,AbsTol) & \multicolumn{3}{|c|}{$(10^{-5},10^{-5})$} \\
            \hline
            $(\mathcal{K}, h_2)$ & \multicolumn{3}{|c|}{$(1,\frac{\pi}{2})$} \\
            \hline
        \end{tabular}
        \caption{Results of Theorem \ref{ThmDepKappaIndepKappaFlat}}
        \label{tabla3}
    \end{center}
\end{table}

Again, as in the previous Theorem, $I^{C}$ is linear in $\mathcal{K}$ and by the same reasoning, we have existence and uniqueness of $K^{*}$.

\end{proof}

\begin{coment}
Notice our choice of the numerical parameters $N, N_2, M$ is not optimal. Smaller parameters might also work, however, as the time required to compute the intervals is not very long, we didn't try to optimize in terms of choosing different values of $N, N_2, M$. In Theorem  \ref{ThmConfinedInhomogeneous}, where the computational costs are higher, we integrate in an adaptive way without fixing the number of points.
\end{coment}

\subsection{The inhomogeneous, confined case}\label{SubsectionProofThm4}

In this subsection we will detail the refinements and technical details that led to the bifurcation diagram shown in Figure \ref{FigBifurcacion}, which illustrates Theorem \ref{ThmConfinedInhomogeneous}.

\subsubsection{Dimension reduction by complex integration}
In \cite{Berselli-Cordoba-GraneroBelinchon:local-solvability-singularities-muskat} the existence of turning singularities is proved by a continuity argument for the full problem \eqref{IIIeqv2}. Here we obtain these turning waves for the full range $|\mathcal{K}|<1$. We will write the equation for the velocity in a more suitable way by calculating explicitly some of the integrals using complex integration. We remark that we are transforming an a priori 4-dimensional problem into a 2-dimensional one, dramatically reducing the resources needed for its computation. We will denote the complex argument function, \emph{i.e.} the function that given a complex number returns its phase, by $\arg(z)$ and consider the branch that takes values in $[-\pi,\pi)$. We start with some useful Lemmas whose proof (similar to the proof of Lemma \ref{lemacitado}) we omit for the sake of brevity:
\begin{lem}\label{IIIlem1}
We have, for $-1 < d < 1$, $c, y \in \mathbb{R}$:
$$
\int_\RR\frac{\cos(y\xi)d\xi}{\cosh(c-\xi)+d}=-\frac{2\pi}{\sqrt{1-d^2}}\frac{\cos(yc)\sinh(y \arg(-d+\sqrt{1-d^2}i)-y\pi)}{\sinh(\pi y)}.
$$
\end{lem}
%\begin{proof}
%To compute the integral we use the Residue Theorem. We write $u=c-\xi$ and $w= e^u$. Then, in order to obtain the poles we consider
%$$
%w^+=-d+i\sqrt{1-d^2}\text{ and } w^-=-d-i\sqrt{1-d^2}.
%$$
%Since $|w^\pm|=1$ we have $w^\pm=e^{i\text{arg}(w^\pm)}$. We obtain $u^\pm_{k}=i\text{arg}(w^\pm)+i2k\pi $ with $k\in\ZZ$. We conclude that the poles are
%$$
%\xi^\pm_{k}=c-u^\pm_{k}=c-i\text{arg}(w^\pm)-i2k\pi.
%$$
%We take the boundary of the rectangle $(c-R,c+R)\times(0,2\pi)$ as the contour. The poles inside this domain are $\xi^-_{0}$ and $\xi^+_{-1}$. Then, applying the Residue Theorem and taking the limit in $R$, we obtain
%\begin{align*}
%2\pi i \text{ Res}\left( \sum \frac{\cos(y\xi)}{\cosh(c-\xi)+d},\xi_{k}^{\pm}\right) & = \int_\RR\frac{\cos(y\xi)d\xi}{\cosh(c-\xi)+d} - \int_\RR\frac{ \cos(2\pi i y +y\xi)d\xi}{\cosh(c-\xi)+d} \\
%& = (1 - \cosh(2\pi y))\int_\RR\frac{\cos(y\xi)d\xi}{\cosh(c-\xi)+d} \\
%& + i \sinh(2\pi y) \int_\RR\frac{\sin(y\xi)d\xi}{\cosh(c-\xi)+d}
%\end{align*}
%Taking real parts on both sides of the equation,
%\begin{multline*}
%\int_\RR\frac{\cos(y\xi)d\xi}{\cosh(c-\xi)+d}=\frac{\re\left(2\pi i\left(\text{Res}\left(\frac{\cos(y\xi)}{\cosh(c-\xi)+d},\xi_{0}^{-}\right) + \text{Res}\left(\frac{\cos(y\xi)}{\cosh(c-\xi)+d},\xi_{-1}^{+}\right)\right)\right)}{1-\cosh(2\pi y)}\\
%=-\frac{2\pi}{\sqrt{1-d^2}}\frac{\cos(yc)\sinh(y \arg(-d+\sqrt{1-d^2}i)-y\pi)}{\sinh(\pi y)}.
%\end{multline*}
%\end{proof}

\begin{lem}\label{IIIlem2}
We have, for $-1 < d < 1$, $b, c,  y \in \mathbb{R}$:
\begin{multline*}
\int_\RR\frac{\cos(y(\xi+c))(-\cosh(\xi)\cos(h_2)+b)d\xi}{(\cosh(\xi)+d)^ 2}\\
=\frac{2\pi}{\sinh(\pi y)}\left(\frac{y\cosh(y\cdot \emph{arg}(-d+i\sqrt{1-d^2})-\pi y)\cos(cy)(d\cos(h_2)+b)}{1-d^2}\right.\\
\left.+\frac{\sinh\left(y\cdot \emph{arg}(-d+i\sqrt{1-d^2})-\pi y\right)\cos(cy)\left(\cos(h_2)+db\right)}{(\sqrt{1-d^2})^3}\right).
\end{multline*}
\end{lem}
\begin{lem}\label{IIIlem3}
We have, for $0<a<\pi\leq c$, $b\in\RR$
$$
\int_\RR\frac{\cos(yb)\sinh(ay)}{\sinh(cy)}dy=\frac{\pi}{c}\frac{\sin\left(\frac{\pi}{c}a\right)}{\cos\left(\frac{\pi}{c}a\right)+\cosh\left(\frac{\pi}{c}b\right)}.
$$
\end{lem}
\begin{proof}
Using classical trigonometric identities, we have,
$$
\cos(yb)\sinh(ay)=\frac{1}{2}\left[\sinh(ay-iby)+\sinh(ay+iby)\right],
$$
so, we need to compute the integral
$$
I_{ss}=\int_{\RR}\frac{\sinh(wy)}{\sinh(cy)}dy,
$$
for the appropriate $w$.
\begin{figure}[h!]\centering
\includegraphics[scale=0.75]{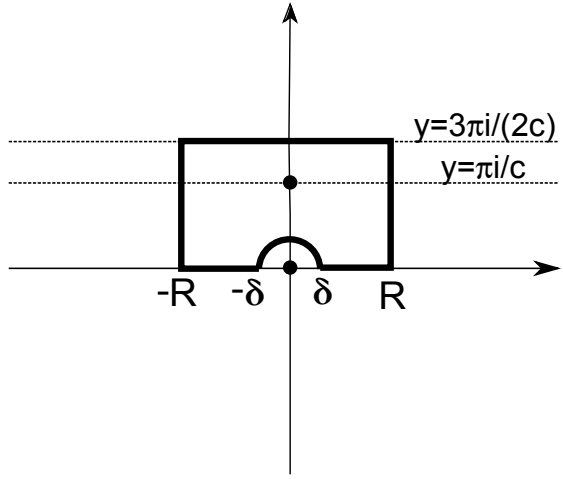}
\caption{Contour of integration.}
\label{poloslema4}
\end{figure}
We use complex integration and the Residue Theorem. We consider the contour given by $\displaystyle\mathcal{C}_1=\cup_i\Gamma_i$ and by $R_1$ the interior region delimited by $\mathcal{C}_1$, where
$$
\Gamma_1=\left\{z=x+\frac{3}{2}\frac{\pi}{c}i, x\in (-R,R)\right\}, \text{ }\Gamma_2=\{z=x, x\in(-R,-\delta)\cup(\delta,R)\},
$$
$$
\Gamma_3=\{z=\delta e^{i\theta}, \theta\in(0,\pi)\}, \Gamma_{4}=\left\{z=\pm R+iy,y\in(0,\frac{3}{2}\frac{\pi}{c})\right\}.
$$
Since $c>a$ and $b$ only deals with oscillations with bounded amplitude, we have
$$
\lim_{R\rightarrow\infty}\int_{\Gamma_4}\frac{\sinh(wz)}{\sinh(cz)}dz=0,
$$
$$
\lim_{\delta\rightarrow 0}\int_{\Gamma_3}\frac{\sinh(wz)}{\sinh(cz)}dz\leq \lim_{\delta\rightarrow 0}C\delta=0.
$$
Thus, we are left with
$$
\lim_{\delta\rightarrow 0, R\rightarrow\infty}\int_{\mathcal{C}_1}\frac{\sinh(wz)}{\sinh(cz)}dz=\int_{-\infty}^\infty\frac{\sinh(wy)}{\sinh(cy)}dy+\int_{-\infty}^\infty\frac{\sinh\left(wy+w\frac{3}{2}\frac{\pi}{c}i\right)}{i\cosh(c y)}dy=2\pi i\text{Res}_{ss},
$$
where
$$
\text{Res}_{ss}=\sum_{\xi\in R_1, \xi \text{ poles}} \text{Res}\left(\frac{\sinh(wz)}{\sinh(cz)},\xi\right)=-\frac{\sinh\left(\frac{w\pi}{c}i\right)}{c}.
$$
We need to compute a helpful integral
$$
I_{cc}=\int_{-\infty}^\infty \frac{\cosh(wy)}{\cosh(cy)}dy.
$$
We define the contour $\mathcal{C}_2$ as the boundary of the rectangle $R_2=[-R,R]\times [0,\frac{\pi}{c}]$. We get
$$
\int_{\mathcal{C}_2}\frac{\cosh(wz)}{\cosh(cz)}dz=\int_{-\infty}^\infty\frac{\cosh(wy)}{\cosh(cy)}dy+\int_{-\infty}^\infty\frac{\cosh\left(wy+w\frac{\pi}{c}i\right)}{\cosh(c y)}dy=2\pi i\text{Res}_{cc},
$$
where
$$
\text{Res}_{cc}=\sum_{\xi\in R_2, \xi \text{ poles}} \text{Res}\left(\frac{\cosh(wz)}{\cosh(cz)},\xi\right)=-\frac{\cosh\left(\frac{w\pi}{2c}i\right)}{ci}.
$$
Using trigonometric identities, we obtain
$$
\left(1+\cos\left(w\frac{\pi}{c}\right)\right)I_{cc}+i\sin\left(w\frac{\pi}{c}\right)\int_{-\infty}^\infty\frac{\sinh\left(wy\right)}{\cosh(c y)}dy=2\pi i\text{Res}_{cc}.
$$
Therefore, using the oddness of the second integrand,
$$
I_{cc}=\frac{\frac{2\pi}{c}\cosh\left(\frac{w\pi i}{2c}\right)}{1+\cos\left(w\frac{\pi}{c}\right)}=\frac{\pi}{c}\frac{1}{\cos\left(w\frac{\pi}{2c}\right)}.
$$
Inserting this value in the previous expression for $I_{ss}$, we obtain
$$
I_{ss}+\sin\left(w\frac{3\pi}{2c}\right)I_{cc}=2\pi i \frac{\sinh\left(w\frac{\pi i}{c}\right)}{-c}=\frac{2\pi\sin\left(w\frac{\pi}{c}\right)}{c},
$$
thus,
$$
I_{ss}=\frac{\pi}{c}\tan\left(w\frac{\pi}{2c}\right).
$$
Finally,
$$
\int_\RR\frac{\cos(yb)\sinh(ay)}{\sinh(cy)}dy=\frac{\pi}{2c}\left(\tan\left((a-ib)\frac{\pi}{2c}\right)+\tan\left((a+ib)\frac{\pi}{2c}\right)\right)
=\frac{\pi}{c}\frac{\sin\left(\frac{\pi}{c}a\right)}{\cos\left(\frac{\pi}{c}a\right)+\cosh\left(\frac{\pi}{c}b\right)}.
$$
\end{proof}
Then, according to Lemmas \ref{IIIlem1} and \ref{IIIlem2}, we have
\begin{multline}\label{IIIa}
\int_\RR\frac{\cos(y\xi)d\xi}{\cosh(-\beta-z_1(\gamma)-\xi)-\cos(h_2+z_2(\gamma))}\\
=\frac{2\pi}{\sin(h_2+z_2(\gamma))}\frac{\cos(y(\beta+z_1(\gamma)))\sinh(y(\pi-h_2-z_2(\gamma)))}{\sinh(\pi y)}
\end{multline}
and
\begin{multline}\label{IIIb}
\int_\RR\frac{\cos(y\xi)d\xi}{\cosh(-\beta-z_1(\gamma)-\xi)+\cos(h_2-z_2(\gamma))}\\
=\frac{2\pi}{\sin(h_2-z_2(\gamma))}\frac{\cos(y(\beta+z_1(\gamma)))\sinh(y(h_2-z_2(\gamma)))}{\sinh(\pi y)},
\end{multline}
\begin{equation}\label{IIIc}
\int_\RR\frac{\cos(y(\beta+z_1(\gamma)))(-\cosh(\beta)\cos(h_2)+1)d\beta}{(\cosh(\beta)-\cos(h_2))^ 2}=2\pi\frac{y\cosh(y (\pi- h_2))\cos(z_1(\gamma)y)}{\sinh(\pi y)}
\end{equation}
and
\begin{multline}\label{IIId}
\int_\RR\frac{\cos(y(\beta+z_1(\gamma)))(-\cosh(\beta)\cos(h_2)-\cos^ 2(h_2)+\sin^2(h_2))d\beta}{(\cosh(\beta)+\cos(h_2))^ 2}\\
=\frac{2\pi \cos(z_1(\gamma)y)}{\sinh(\pi y)}\left(y\cosh(y h_2)-\frac{2 \sinh\left(y h_2\right)}{\tan(h_2)}\right).
\end{multline}

In particular, using Lemma \ref{IIIlem3},
\begin{multline}\label{eqaux}
\int_\RR\frac{\cos(y(\beta+z_1(\gamma)))\sinh((\pi-h_2-z_2(\gamma))y)}{\sinh(\pi y)}dy\\
=\frac{\sin\left(\pi-h_2-z_2(\gamma)\right)}{\cos\left(\pi-h_2-z_2(\gamma)\right)+\cosh\left(\beta+z_1(\gamma)\right)}\\
=\frac{\sin\left(h_2+z_2(\gamma)\right)}{-\cos\left(h_2+z_2(\gamma)\right)+\cosh\left(\beta+z_1(\gamma)\right)},
\end{multline}
and
\begin{multline}\label{eqaux2}
\int_\RR\frac{\cos(y(\beta+z_1(\gamma)))\sinh((-h_2+z_2(\gamma))y)}{\sinh(\pi y)}dy=\frac{\sin\left(-h_2+z_2(\gamma)\right)}{\cos\left(-h_2+z_2(\gamma)\right)+\cosh\left(\beta+z_1(\gamma)\right)}.
\end{multline}

We proceed now to calculate $\da v_1(0)$. We fix $\bar{\rho} = 1$. Then, the appropriate expression is
$$
\paa v_1(0)=\pat\paa z_1(0)=I_1+I_2,
$$
where
$$
I_1=2\paa z_2(0)\int_0^\infty \frac{\paa z_1(\beta)\sinh(z_1(\beta))\sin(z_2(\beta))}{\left(\cosh(z_1(\beta))-\cos(z_2(\beta))\right)^2}+\frac{\paa z_1(\beta)\sinh(z_1(\beta))\sin(z_2(\beta))}{\left(\cosh(z_1(\beta))+\cos(z_2(\beta))\right)^2}d\beta,
$$
and
\begin{multline*}
I_2=\frac{\paa z_2(0)}{4\pi}\text{P.V.}\int_\RR\frac{\varpi_2(-\beta)(-\cosh(\beta)\cos(h_2)+1)}{(\cosh(\beta)-\cos(h_2))^2}d\beta\\
+\frac{\paa z_2(0)}{4\pi}\text{P.V.}\int_\RR\frac{\varpi_2(-\beta)(-\cosh(\beta)\cos(h_2)-\cos^2(h_2)+\sin^2(h_2))}{(\cosh(\beta)+\cos(h_2))^2}d\beta,
\end{multline*}
where $\varpi_2$ is given in \eqref{IIIw2defb}.
%\begin{eqnarray*}
%\varpi_2(x)&=&2\mathcal{K}\text{P.V.}\int_\RR\frac{\paa z_2(\gamma)\sin(h_2+z_2(\gamma))}{\cosh(x-z_1(\gamma))-\cos(h_2+z_2(\gamma))}d\gamma\nonumber\\
%&&-2\mathcal{K}\text{P.V.}\int_\RR\frac{\paa z_2(\gamma)\sin(-h_2+z_2(\gamma))}{\cosh(x-z_1(\gamma))+\cos(-h_2+z_2(\gamma))}d\gamma\nonumber\\
%&&-2\frac{\mathcal{K}^2}{2\pi}G_{h_2,\mathcal{K}}*\text{P.V.}\int_\RR\frac{\paa z_2(\gamma)\sin(h_2+z_2(\gamma))}{\cosh(x-z_1(\gamma))-\cos(h_2+z_2(\gamma))}d\gamma\nonumber\\
%&&+2\frac{\mathcal{K}^2}{2\pi}G_{h_2,\mathcal{K}}*\text{P.V.}\int_\RR\frac{\paa z_2(\gamma)\sin(-h_2+z_2(\gamma))}{\cosh(x-z_1(\gamma))+\cos(-h_2+z_2(\gamma))}d\gamma.
%\end{eqnarray*}
Now we use Lemmas \ref{IIIlem1} and \ref{IIIlem2} to compute explicitly some of the integrals in $I_2$. Notice that the space is $\sigma-$finite and, taking the absolute value, we can apply Tonelli-Fubini Theorem. First, we integrate in $\xi$ using \eqref{IIIa} and \eqref{IIIb}, and by Lemma \ref{IIIlem3} and equations \eqref{eqaux} and \eqref{eqaux2}, we obtain
\begin{align*}
I_2=&\frac{\paa z_2(0)\mathcal{K}}{2\pi}\bigg{[}\text{P.V.}\int_\RR\int_\RR\text{P.V.}\int_\RR\frac{
\frac{\paa z_2(\gamma)\sinh(y(\pi-h_2-z_2(\gamma)))}{\sinh(\pi y)+\mathcal{K}\sinh(2h_2 y)}\cos(y(\beta+z_1(\gamma))) }{(\cosh(\beta)-\cos(h_2))^2 (-\cosh(\beta)\cos(h_2)+1)^{-1}}dy d\gamma d\beta\\
&+\text{P.V.}\int_\RR\int_\RR\text{P.V.}\int_\RR\frac{\frac{\paa z_2(\gamma)\sinh(y(\pi-h_2-z_2(\gamma)))}{\sinh(\pi y)+\mathcal{K}\sinh(2h_2 y)}\cos(y(\beta+z_1(\gamma)))}{(\cosh(\beta)+\cos(h_2))^2(-\cosh(\beta)\cos(h_2)-\cos(2h_2))^{-1}}dy d\gamma  d\beta\\
&+\text{P.V.}\int_\RR\int_\RR\text{P.V.}\int_\RR\frac{\frac{\paa z_2(\gamma)\sinh(y(h_2-z_2(\gamma)))}{\sinh(\pi y)+\mathcal{K}\sinh(2h_2 y)}\cos(y(\beta+z_1(\gamma)))}{(\cosh(\beta)-\cos(h_2))^2(-\cosh(\beta)\cos(h_2)+1)^{-1}} dy d\gamma  d\beta\\
&+\text{P.V.}\int_\RR\int_\RR\text{P.V.}\int_\RR\frac{\frac{\paa z_2(\gamma)\sinh(y(h_2-z_2(\gamma)))}{\sinh(\pi y)+\mathcal{K}\sinh(2h_2 y)}\cos(y(\beta+z_1(\gamma)))}{(\cosh(\beta)+\cos(h_2))^2(-\cosh(\beta)\cos(h_2)-\cos(2h_2))^{-1}}dy d\gamma d\beta\bigg{]}.
\end{align*}

Now we integrate in $\beta$ using \eqref{IIIc} and \eqref{IIId}:
%\begin{align*}
%I_2=&\paa z_2(0)\mathcal{K}\bigg{[}\int_\RR\text{P.V.}\int_\RR
%\frac{\paa z_2(\gamma)\sinh(y(\pi-h_2-z_2(\gamma)))y}{(\sinh(\pi y)+\mathcal{K}\sinh(2h_2 y))\sinh(\pi y)}\\
%&\hspace{5cm}\times\cosh(y(\pi-h_2))\cos(yz_1(\gamma)) d\gamma dy\\
%&+\int_\RR\text{P.V.}\int_\RR\frac{\paa z_2(\gamma)\sinh(y(\pi-h_2-z_2(\gamma)))}{(\sinh(\pi y)+\mathcal{K}\sinh(2h_2 y))\sinh(\pi y)}\\
%&\hspace{5cm}\times\cos(z_1(\gamma)y)\left(y\cosh(yh_2)-\frac{2\sinh(yh_2)}{\tan(h_2)}\right)d\gamma dy\\
%&+\int_\RR\text{P.V.}\int_\RR\frac{\paa z_2(\gamma)\sinh(y(h_2-z_2(\gamma)))y}{(\sinh(\pi y)+\mathcal{K}\sinh(2h_2 y))\sinh(\pi y)}\\
%&\hspace{5cm}\times\cosh(y(\pi-h_2))\cos(yz_1(\gamma)) d\gamma dy\\
%&+\int_\RR\text{P.V.}\int_\RR\frac{\paa z_2(\gamma)\sinh(y(h_2-z_2(\gamma)))}{(\sinh(\pi y)+\mathcal{K}\sinh(2h_2 y))\sinh(\pi y)}\\
%&\hspace{5cm}\times\cos(z_1(\gamma)y)\left(y\cosh(yh_2)-\frac{2\sinh(yh_2)}{\tan(h_2)}\right)d\gamma dy\bigg{]}.
%\end{align*}
%Factoring out the integrands, we get
\begin{multline*}
I_2=\paa z_2(0)\mathcal{K}\int_\RR\text{P.V.}\int_\RR
\frac{\paa z_2(\gamma)\cos(yz_1(\gamma))}{(\sinh(\pi y)+\mathcal{K}\sinh(2h_2 y))\sinh(\pi y)}\\
\times\left(\sinh(y(\pi-h_2-z_2(\gamma)))+\sinh(y(h_2-z_2(\gamma)))\right)\\
\times \left(y\cosh(y(\pi-h_2))+y\cosh(yh_2)-\frac{2\sinh(yh_2)}{\tan(h_2)}\right) d\gamma dy.
\end{multline*}
Using the oddness of $z_i$, we obtain
\begin{multline*}
I_2=4\paa z_2(0)\mathcal{K}\int_0^\infty\int_0^\infty
\frac{\paa z_2(\gamma)\cos(yz_1(\gamma))}{(\sinh(\pi y)+\mathcal{K}\sinh(2h_2 y))\sinh(\pi y)}\\
\times\left(2\sinh\left(y\frac{\pi}{2}\right)\cosh\left(y z_2(\gamma)\right)\cosh\left(y\left(\frac{\pi}{2}-h_2\right)\right)\right)\\
\times \left(y\cosh(y(\pi-h_2))+y\cosh(yh_2)-\frac{2\sinh(yh_2)}{\tan(h_2)}\right) d\gamma dy,
\end{multline*}
and, using trigonometrical identities, we get the final expression
\begin{multline*}
I_2=4\paa z_2(0)\mathcal{K}\int_0^\infty\int_0^\infty
\frac{\paa z_2(\gamma)\cos(z_1(\gamma)y)}{(\sinh(\pi y)+\mathcal{K}\sinh(2h_2 y))\cosh\left(y\frac{\pi}{2}\right)}\\
\times\left(2y\cosh\left(\frac{y\pi}{2}- y h_2\right)\cosh\left(\frac{y \pi}{2}\right)-\frac{2\sinh\left(y h_2\right)}{\tan(h_2)}\right)\\
\times \cosh\left(y z_2(\gamma)\right)\cosh\left(y\left(\frac{\pi}{2}-h_2\right)\right) d\gamma dy.
\end{multline*}

%\subsubsection{Implementation details}

%\begin{figure}[h!]\centering
%\includegraphics[trim=3cm 0mm 0mm 0mm,scale=0.5]{pictures/BifurcacionRY.eps}
%\caption{Bifurcation diagram BLABLABLA}
%\label{FigBifurcacion}
%\end{figure}

%\begin{figure}[h!]\centering
%\includegraphics[scale=0.5]{pictures/BifurcacionBW.eps}
%\caption{Bifurcation diagram BLABLABLA}
%\label{FigBifurcacion}
%\end{figure}

\subsubsection{Technical details concerning Theorem \ref{ThmConfinedInhomogeneous}(a)}

%\begin{figure}[h!]\centering
%\includegraphics[trim=2cm 0mm 0mm 0mm,scale=0.75]{pictures/BifurcacionOKRY.eps}
%\caption{Bifurcation diagram corresponding to the phenomenon of turning/not turning for the initial condition given by the family of curves \eqref{initialConditionBif}. Yellow: not turning, red: turning.}
%\label{FigBifurcacion}
%\end{figure}

The first four statements of Theorem \ref{ThmConfinedInhomogeneous} can be deduced from Figure \ref{FigBifurcacion}. In this section, we explain the algorithms and the technical details that led us to the rigorous computation of the bifurcation diagram.

We have been more careful in the optimization of the codes concerning the diagram since we were expecting a higher computation time. However, the possibility of parallelization did not force us to optimize up to a very low level, just to simply maintain the correct complexity of the code. The implementation is now split into several files, and many of the headers of the functions (such as the integration methods) contain pointers to functions (the integrands) so that they can be reused for an arbitrary amount of integrals with minimal changes and easy and safe debugging.

The initial condition family we used for the bifurcation diagram was
\begin{align}
z_1(\al) & = \al - \sin(\al)e^{-K \al^{2}}, \quad K = 10^{-4} \nonumber \\
z_2(\al) & = h_2\frac{3}{\pi}\left(\frac{\sin(3\al)}{3}- \frac{\sin(\al)}{2.5}\left(e^{-(\al+2)^2}+e^{-(\al-2)^2}\right)\right)1_{\{|\al| \leq \pi\}}. \label{initialConditionBif}
\end{align}

It is easy to check that $z$ is odd, $z_1(\al)$ is strictly monotone except at $\al = 0$, $\da z_2(0) > 0$ and $|z_2| < h_2$. A more precise bound is given in Lemma \ref{LemmaBoundZ2}. We will compute the bifurcation diagram in the region $(h_2,\mathcal{K}) = \left[\frac{1}{4},\frac{5}{4}\right] \times \left[-1,1\right]$.

The algorithm for the computation of the bifurcation diagram is as follows: we define a structure called \texttt{ParameterSet}, which encapsulates all the necessary information about the parameters and the information needed by the integration procedures in order to compute $\da v_1(0,0)$ for those parameters. More precisely, a ParameterSet contains:
\begin{itemize}
\item Two intervals, \texttt{Left} and \texttt{Right}, which set the limits for the bounded, singularity and unbounded regions (i.e. singularity $ = [0,\text{Left}]$, bounded $= [\text{Left,Right}]$, unbounded $ = [\text{Right},\infty)$, whenever they make sense). In our proof, Left $=0.125$, Right = $16.125$.
\item Two doubles, \texttt{AbsTol} and \texttt{RelTol}, which limit the precision up to which the integrals are computed. In our proof, AbsTol $=$ RelTol $= 10^{-5}$.
\item Two intervals, \texttt{$h_2$} and \texttt{Kappa}, which are the rectangle in the parameter space we are calculating.
\end{itemize}

We mantain a queue (implemented using the Standard Template Library (STL) \texttt{Queue}), in which we store all the ParameterSets to be computed. While the queue is not empty, we take the top element, pop it and give an enclosure of $\da v_1(0,0)$ for this region. Three different cases arise:
\begin{itemize}
\item The enclosure is positive.
\item The enclosure is negative.
\item We can not say anything about its positivity.
\end{itemize}

In both the first two cases, the result is output to its corresponding file (one for the regions for which there is a turning singularity, another for the ones for which there is not). If, on the contrary, we are in the third case, the ParameterSet is split into other narrower ParameterSets which are pushed in the queue. This splitting is only done if the dimensions (both in $h_2$ and $\mathcal{K}$) are bigger than a given limit, which in our case was set to $5 \cdot 10^{-3}$ for the 2 parameters. Moreover, the splitting is not done in a uniform way. We found heuristically that a splitting that cut in 4 in the $h_2$-dimension and in 2 in the $\mathcal{K}$-dimension balanced the width of $I_1$ and $I_2$. If the parameter interval is too narrow, we output the result to a third file, which accounts for the unknown regions.

$I_1$ is split into two parts as in the discussion from Subsection \ref{SubsectionProofsThm23}: a bounded one and a singularity one. The bounded part is calculated using a Gauss-Legendre quadrature of order 2, given by
\begin{multline*}
\int_{a}^{b} f(\eta) d\eta \in  \frac{b-a}{2}\left(f\left(\frac{b-a}{2}\frac{\sqrt{3}}{3} + \frac{b+a}{2}\right)\right.\\\left.+f\left(-\frac{b-a}{2}\frac{\sqrt{3}}{3}+ \frac{b+a}{2}\right)\right)
+\frac{1}{4320}(b-a)^{5}f^{4}([a,b]).
\end{multline*}

Other quadratures of several orders (Gauss-Legendre, Newton-Cotes) were tested and they resulted either in worse results or similar results but worse runtime performance. Moreover, the integration was done in an adaptive way. For each region, we accepted or rejected the result depending on the width in an absolute and a relative way. It is important to notice that because of the uncertainty of the parameters, division by zero is easy to find, even in small integration intervals. In such cases, bisection in the parameter space is needed. We developed extra mechanisms to take care of these cases and discard a ParameterSet once a division by zero is found.

The number of levels of subdivision was also limited, since the uncertainty of the parameters might yield wide enclosures of the integral even with infinite precision. In our case, the maximum number of subdivisions for a non-singular one-dimensional integral was 18, totaling a maximum number of subintervals equal to $2^{18}$, which can be carried out roughly under 90 seconds. Another feature of the integration method is that instead of subdividing the integration intervals by the midpoint (in other words, by the arithmetic mean of the endpoints), we subdivided by the \emph{geometric} mean of the endpoints. While the arithmetic division minimizes the length of the longest piece after the division, the geometric one minimizes the piece with the biggest ratio between its endpoints. This can be particularly useful in many cases: for example in order to avoid divisions by zero for integrands of the type $\frac{1}{\sinh(ay)-\sinh(by)}$, which is the case of $I_2$. However, the geometric division also performs better for $I_1$ and we bisect using that method.

The singular part of $I_1$ was also integrated and not bounded as in the previous sections.  In this case, the algorithm works as follows: we perform Taylor series of order 6 and 4 respectively of the numerator and the denominator and integrate as in \eqref{pax4D}. Potentially this could fail because the uncertainty in $h_2$ (and therefore in $z_2$) could yield a Taylor series in which $0$ belongs to $\da^{4} \mathcal{D}([0,\varepsilon])$. Whenever this happens, we try to integrate using a Gauss-Legendre quadrature of order 2. The integration division is in this case arithmetic since 0 belongs to our integration domain. The maximum subdivision level was set to 12 ($2^{12}$ intervals). If even the Gauss-Legendre quadrature fails, then we return an error and bisect in the space of parameters.

We note that both the singular and the bounded part of $I_1$ are independent of $\mathcal{K}$. For performance purposes, we kept two STL \texttt{Map<ParameterSet,\-Interval,\allowbreak ComparisonFunction>}, with a ComparisonFunction that only sorts by $h_2$ so that for a ParameterSet, we check if we have calculated the values of $I_1$ for that $h_2$ before. If not, once we calculate them we store them in the map.

Regarding $I_2$, we divide it into three regions: singularity, bounded and unbounded. The bounded region is calculated using a 2 dimensional Gauss-Legendre quadrature and geometric division of the subintervals. Here it is clear the need of this subdivision because we want to avoid division by zero and
\begin{align*}
0 \in \sinh(\pi Y) - \sinh(2 h_2 Y) \Leftrightarrow \frac{\sup(Y)}{\inf(Y)} \geq \frac{\pi}{2h_2},
\end{align*}
so the objective is to keep the quotient $\frac{\sup(I)}{\inf(I)}$ as small as possible for every integration interval $I$. The maximum number of subdivision levels was set to 8 ($2^{16}$ rectangles) and the computation time of the bounded part of $I_2$ was well under the 2 minute mark.

For the singularity part, we took the intersection between the interval computed by Gauss-Legendre integration and Taylor expansions (in this case, the expansion is of order 1 in both the numerator and the denominator).
The maximum number of subdivision levels was 7 ($2^{14}$ rectangles).

We end the discussion with the estimations of the unbounded region
\begin{multline*}
I_2^{ub}=4\paa z_2(0)|\mathcal{K}|\int_M^\infty\bigg{(}\int_0^\pi\bigg{|}
\frac{\paa z_2(\gamma)\cos(z_1(\gamma)y)}{(\sinh(\pi y)+\mathcal{K}\sinh(2h_2 y))\cosh\left(y\frac{\pi}{2}\right)}\\
\times\left(2y\cosh\left(\frac{y\pi}{2}- y h_2\right)\cosh\left(\frac{y \pi}{2}\right)-\frac{2\sinh\left(y h_2\right)}{\tan(h_2)}\right)\\
\times \cosh\left(y z_2(\gamma)\right)\cosh\left(y\left(\frac{\pi}{2}-h_2\right)\right)\bigg{|} d\gamma \bigg{)} dy.
\end{multline*}

We will bound the tails using the following inequalities, which are very easy to check,

\begin{align*}
\frac{1}{2} e^{x} \leq \cosh(x) \leq e^{x} \\
\frac{1}{4} e^{x} \leq \sinh(x) \leq \frac{1}{2} e^{x}, \quad x \geq \log(2).
\end{align*}

Now we can show the following naive bounds

\begin{align*}
4\cosh\left(y z_2(\gamma)\right)\cosh\left(y\left(\frac{\pi}{2}-h_2\right)\right)
& \leq 4 e^{-y\left(h_2-\frac{\pi}{2}-\|z_2\|_{L^{\infty}}\right)} \\
\left|\frac{2y\cosh\left(\frac{y\pi}{2}- y h_2\right)\cosh\left(\frac{y \pi}{2}\right)-\frac{2\sinh\left(y h_2\right)}{\tan(h_2)}}{\cosh\left(y\frac{\pi}{2}\right)}\right|
& \leq 2e^{y\left(\frac{\pi}{2}-h_2\right)}y+2\frac{e^{y\left(h_2-\frac{\pi}{2}\right)}}{\tan\left(h_2\right)}.
\end{align*}

For the last factor, we distinguish two cases:

\begin{multline}\label{bichaco}
\frac{1}{(\sinh(\pi y)+\mathcal{K}\sinh(2h_2 y))} \leq
\left\{
\begin{array}{cc}
\displaystyle\frac{2e^{-\pi y}}{1-e^{-2\pi M}}, & \text{ if } \mathcal{K} \geq 0 \\
\displaystyle e^{-\pi y} \frac{2}{1-e^{-2\pi M} - |\mathcal{K}|e^{-M(\pi - 2h_2)}}, & \text{ if } \mathcal{K} < 0 \\
\end{array}
\right\}\\
\equiv e^{-\pi y}C(M,\mathcal{K},h_2),
\end{multline}
where we have used
\begin{align*}
\frac{1}{\sinh(\pi y)} = \frac{2}{e^{\pi y} - e^{-\pi y}} = e^{-\pi y}\frac{2}{1 - e^{-2\pi y}}
\leq e^{-\pi y}\frac{2}{1 - e^{-2\pi M}}.
\end{align*}
Putting all the estimates together, we need to integrate in $y$ and we get

\begin{align*}
& 8 \max_{\mathcal{K}, h_2}\{C(M,\mathcal{K},h_2)\}\int_{M}^{\infty}\left(\frac{e^{-y(\pi - \|z_2\|_{L^{\infty}})}}{\tan(h_2)}+e^{-y(2h_2-\|z_2\|_{L^{\infty}})}y\right)dy \\
 = & 8 \max_{\mathcal{K},h_2}\{C(M,\mathcal{K},h_2)\}\frac{e^{-M(\pi - \|z_2\|_{L^{\infty}})}}{\tan(h_2)\left(\pi - \|z_2\|_{L^{\infty}}\right)} \\
 + & 8 \max_{\mathcal{K},h_2}\{C(M,\mathcal{K},h_2)\}\frac{e^{-M(2h_2 - \|z_2\|_{L^{\infty}})}}{2h_2 - \|z_2\|_{L^{\infty}}}\left(M + \frac{1}{2h_2-\|z_2\|_{L^{\infty}}}\right)
\end{align*}

Finally, we can bound $\|z_2\|_{L^{\infty}}$ in terms of $h_2$ in the following way:

\begin{lem}
\label{LemmaBoundZ2}
Let $z_2(\al)$ be
\begin{align*}
z_2(\al) & = \frac{3}{\pi}\left(\frac{\sin(3\al)}{3}- \frac{\sin(\al)}{2.5}\left(e^{-(\al+2)^2}+e^{-(\al-2)^2}\right)\right)1_{\{|\al| \leq \pi\}}
\end{align*}
Then $\|z_2\|_{L^{\infty}} < 0.65$.
\end{lem}
\begin{proof}
The proof is computer-assisted and the code can be found in the supplementary material. The algorithm is the classical branch and bound \cite{Neumaier:branch-and-bound}: given an interval $I$ we first compute an enclosure $z_2(I)$. If the diameter is not small enough (smaller than a given tolerance), we split $I$ into $I^{L}, I^{R}$ such that $I \subset I^{L} \cup I^{R}$ and call the same function to get their $L^{\infty}$ norms recursively. We merge the results using that
\begin{align*}
\|z_2\|_{L^{\infty}(I)} \subset \max\left\{\|z_2\|_{L^{\infty}(I^{L})}, \|z_2\|_{L^{\infty}(I^{R})}\right\},
\end{align*}
 where the $\max$ operation between intervals was defined in \eqref{defMAX}. For a tolerance equal to $2 \cdot10^{-6}$, our program outputs the following bound:
\begin{align*}
\|z_2\|_{L^{\infty}} \in 0.64627^{3239}_{4666}.
\end{align*}
This proves the Lemma.
\end{proof}
Thus, we can bound the contribution of the unbounded part $I_{2}^{ub}$ by
\begin{multline*}
\paa z_2(0)|\mathcal{K}| \int_{0}^\pi |\partial_{\gamma}   z_2(\gamma)| d\gamma\bigg{(}
8 \max_{\mathcal{K},h_2}\{C(M,\mathcal{K},h_2)\}\frac{e^{-M(\pi - 0.65 h_2})}{\tan(h_2)\left(\pi - 0.65h_2\right)} \\
 +  8 \max_{\mathcal{K},h_2}\{C(M,\mathcal{K},h_2)\}\frac{e^{-M(1.35h_2})}{1.35h_2}\left(M + \frac{1}{1.35h_2}\right)\bigg{)}.
\end{multline*}

For the computation of the integral of $|\partial_\gamma z_2|$, we note that this integral is linear in $h_2$ (since it is linear in $z_2$) and we use an unnormalized version of $z_2$, namely $\frac{z_2}{h_2}$ and multiply by $h_2$ at the end. This narrows the resulting interval.

We computed the bifurcation diagram depicted in Figure \ref{FigBifurcacion}. We could give an answer regarding the question of turning or not to $97.14\%$ of the parameter space. $53.23\%$ of the space turned (red) and $43.91\%$ did not turn (yellow). The remaining $2.86\%$ is painted in white. The computation was done in parallel (every core was allocated an initial region) over 8 cores. The division along the cores was made in such a way that core $i = 1,\ldots,8$ started to compute the region $\left[\frac{1}{4},\frac{5}{4}\right] \times \left[-1 + \frac{i-1}{4}, -1 + \frac{i}{4}\right]$. The average runtime was about 30 hours per core, and a total of 5960 rectangles were calculated (an average of 2.5 minutes per rectangle): 8 of the first generation, 64 of the second, 512 of the third, 1880 of the fourth and 3496 of the fifth, out of which 1871 gave a positive result (not turning), 2407 gave a negative (turning) and the rest did not give an answer to the sign and were subdivided or output to a file depending on their width.

\begin{figure}[h!]\centering
\includegraphics[trim=2cm 0mm 0mm 0mm,scale=0.55]{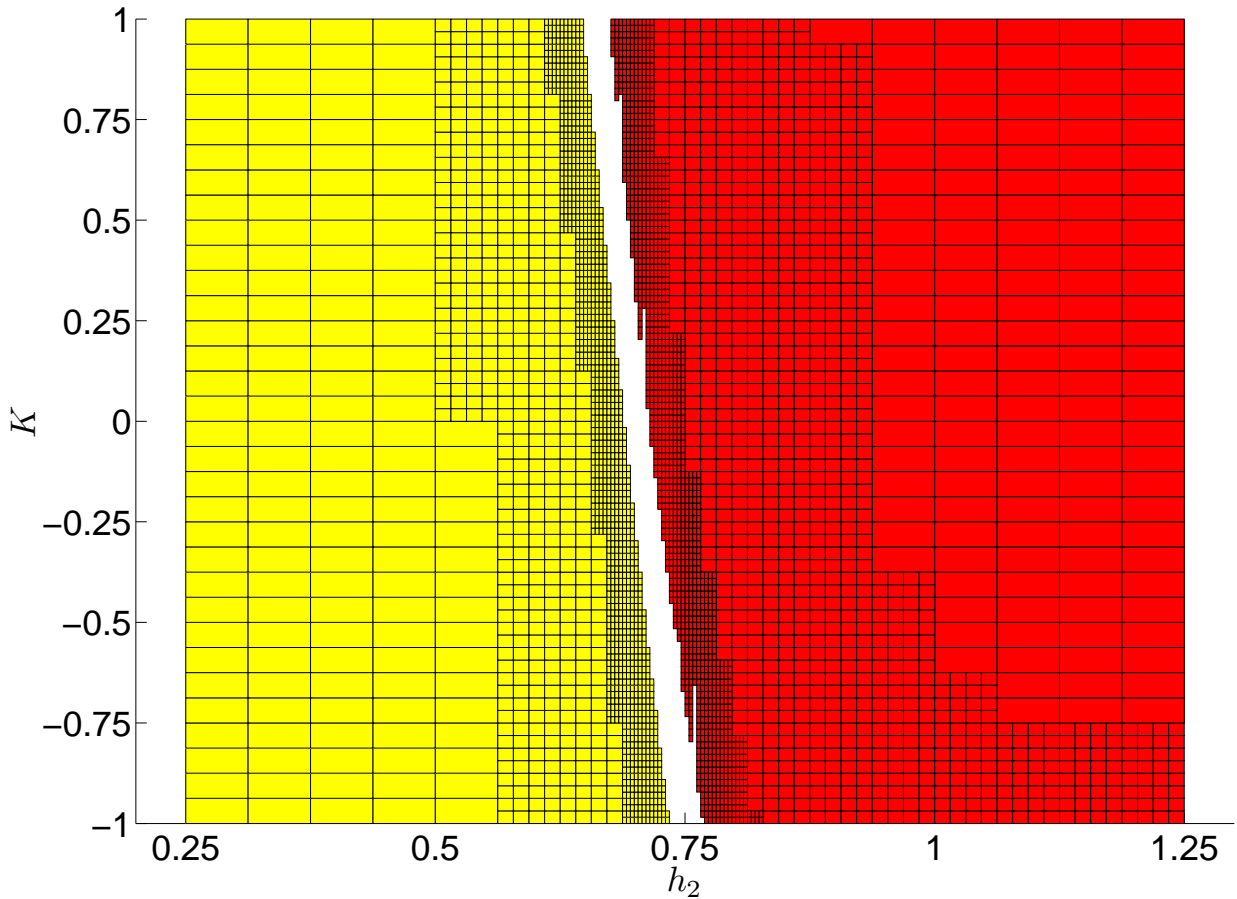}
\caption{Bifurcation diagram corresponding to the phenomenon of turning/not turning for the initial condition given by the family of curves \eqref{initialConditionBif}. Yellow (lighter color): not turning, red (darker color): turning.}
\label{FigBifurcacion}
\end{figure}

\subsubsection{Technical details concerning Theorem \ref{ThmConfinedInhomogeneous}(b)}
We want to invoke the Implicit Function Theorem. Thus, we have to check that 
$$
\frac{d}{d\mathcal{K}}\pat\paa z_1(0,0)\neq0\text{ for points }(h_2,\mathcal{K}) \text{ such that }\pat\paa z_1(0,0)=0.
$$
In particular, we have to check the previous condition in an open set containing the white region in Figure \ref{FigBifurcacion}. We compute
\begin{multline*}
DI_2 \equiv \frac{d}{d\mathcal{K}}\pat\paa z_1(0,0)=4\paa z_2(0)\int_0^\infty\int_0^\infty
\frac{\sinh(\pi y)\paa z_2(\gamma)\cos(z_1(\gamma)y)}{(\sinh(\pi y)+\mathcal{K}\sinh(2h_2 y))^2\cosh\left(y\frac{\pi}{2}\right)}\\
\times\left(2y\cosh\left(\frac{y\pi}{2}- y h_2\right)\cosh\left(\frac{y \pi}{2}\right)-\frac{2\sinh\left(y h_2\right)}{\tan(h_2)}\right)\\
\times \cosh\left(y z_2(\gamma)\right)\cosh\left(y\left(\frac{\pi}{2}-h_2\right)\right) d\gamma dy.
\end{multline*}

As in Theorem \ref{ThmConfinedInhomogeneous}(a), we divide the integral into three diferent regions: singularity, bounded and unbounded, which are calculated in the same way as for the previous Theorem. All what is left is to estimate the tails.

Using \eqref{bichaco}, we have
$$
\frac{\sinh(\pi y)}{\sinh(\pi y)+\mathcal{K}\sinh(2h_2 y)}\leq \sinh(\pi y)e^{-\pi y}C(M,\mathcal{K},h_2)\leq \frac{C(M,\mathcal{K},h_2)}{2}.
$$
With the previous estimates we have that the tail contribution can be bounded by
\begin{multline*}
\paa z_2(0) \int_{0}^\pi |\partial_{\gamma}   z_2(\gamma)| d\gamma\bigg{(}
4 \max_{\mathcal{K},h_2}\{C(M,\mathcal{K},h_2)\}^2\frac{e^{-M(\pi - 0.65 h_2})}{\tan(h_2)\left(\pi - 0.65h_2\right)} \\
 +  4 \max_{\mathcal{K},h_2}\{C(M,\mathcal{K},h_2)\}^2\frac{e^{-M(1.35h_2})}{1.35h_2}\left(M + \frac{1}{1.35h_2}\right)\bigg{)}.
\end{multline*}

Again, the computation was split among 8 cores, which took as input the intervals output as ``unknown'' in Theorem \ref{ThmConfinedInhomogeneous}(a) and ran for about 4 hours. All of them verified a negative sign for $DI_2$ in those intervals, without needing to split them into further subintervals.

\bibliographystyle{abbrv}
\bibliography{bibliografia,references}

\begin{thebibliography}{10}

\bibitem{ambrose2004well}
D.~Ambrose.
\newblock {W}ell-posedness of two-phase {H}ele-{S}haw flow without surface
  tension.
\newblock {\em {E}uropean {J}ournal of {A}pplied {M}athematics},
  15(5):597--607, 2004.

\bibitem{Berselli-Cordoba-GraneroBelinchon:local-solvability-singularities-mus%
kat}
L.~Berselli, D.~C\'ordoba, and R.~Granero-Belinch\'on.
\newblock Local solvability and finite time singularities for the inhomogeneous
  {M}uskat problem.
\newblock To appear, 2013.

\bibitem{Berz-Makino:high-dimensional-quadrature}
M.~Berz and K.~Makino.
\newblock New methods for high-dimensional verified quadrature.
\newblock {\em Reliable Computing}, 5(1):13--22, 1999.

\bibitem{Bona-Lannes-Saut:asymptotic-internal-waves}
J.~L. Bona, D.~Lannes, and J.-C. Saut.
\newblock Asymptotic models for internal waves.
\newblock {\em J. Math. Pures Appl. (9)}, 89(6):538--566, 2008.

\bibitem{castro2012breakdown}
A.~Castro, D.~Cordoba, C.~Fefferman, and F.~Gancedo.
\newblock {B}reakdown of smoothness for the {M}uskat problem.
\newblock {\em {A}rchive for {R}ational {M}echanics and {A}nalysis},
  208(3):805--909, 2013.

\bibitem{Castro-Cordoba-Fefferman-Gancedo-GomezSerrano:finite-time-singulariti%
es-water-waves-surface-tension}
A.~Castro, D.~C{\'o}rdoba, C.~Fefferman, F.~Gancedo, and J.~G{\'o}mez-Serrano.
\newblock Finite time singularities for water waves with surface tension.
\newblock {\em Journal of Mathematical Physics}, 53(11):115622--115622, 2012.

\bibitem{Castro-Cordoba-Fefferman-Gancedo-GomezSerrano:splash-water-waves}
A.~Castro, D.~C{\'o}rdoba, C.~Fefferman, F.~Gancedo, and J.~G{\'o}mez-Serrano.
\newblock Splash singularity for water waves.
\newblock {\em Proceedings of the National Academy of Sciences},
  109(3):733--738, 2012.

\bibitem{Castro-Cordoba-Fefferman-Gancedo-GomezSerrano:finite-time-singulariti%
es-free-boundary-euler}
A.~Castro, D.~C{\'o}rdoba, C.~Fefferman, F.~Gancedo, and J.~G{\'o}mez-Serrano.
\newblock Finite time singularities for the free boundary incompressible
  {E}uler equations.
\newblock {\em Ann. of Math. (2)}, 178(3):1061--1134, 2013.

\bibitem{ccfgl}
A.~Castro, D.~Cordoba, C.~Fefferman, F.~Gancedo, and M.~Lopez-Fernandez.
\newblock {R}ayleigh-{T}aylor breakdown for the {M}uskat problem with
  applications to water waves.
\newblock {\em Annals of Math}, 175:909--948, 2012.

\bibitem{Castro-Cordoba-Gancedo:recent-results-muskat}
A.~Castro, D.~C{\'o}rdoba, and F.~Gancedo.
\newblock Some recent results on the {M}uskat problem.
\newblock {\em Journ\'ees Equations aux Derivees Partielles}, (5), 2010.

\bibitem{CF}
M.~Cerminara and A.~Fasano.
\newblock {M}odelling the dynamics of a geothermal reservoir fed by gravity
  driven flow through overstanding saturated rocks.
\newblock {\em Journal of Volcanology and Geothermal Research}, 233:37--54,
  2012.

\bibitem{ccgs-10}
P.~Constantin, D.~Cordoba, F.~Gancedo, and R.~Strain.
\newblock On the global existence for the {M}uskat problem.
\newblock {\em Journal of the European Mathematical Society}, 15:201--227,
  2013.

\bibitem{c-c-g10}
A.~Cordoba, D.~C{\'o}rdoba, and F.~Gancedo.
\newblock {I}nterface evolution: the {H}ele-{S}haw and {M}uskat problems.
\newblock {\em Annals of Math}, 173, no. 1:477--542, 2011.

\bibitem{c-g07}
D.~C{\'o}rdoba and F.~Gancedo.
\newblock {C}ontour dynamics of incompressible 3-{D} fluids in a porous medium
  with different densities.
\newblock {\em Communications in Mathematical Physics}, 273(2):445--471, 2007.

\bibitem{c-g09}
D.~C{\'o}rdoba and F.~Gancedo.
\newblock {A} maximum principle for the {M}uskat problem for fluids with
  different densities.
\newblock {\em Communications in Mathematical Physics}, 286(2):681--696, 2009.

\bibitem{c-g-o08}
D.~C{\'o}rdoba, F.~Gancedo, and R.~Orive.
\newblock {A} note on interface dynamics for convection in porous media.
\newblock {\em Physica D: Nonlinear Phenomena}, 237(10-12):1488--1497, 2008.

\bibitem{Cordoba-GraneroBelinchon-Orive:confined-muskat}
D.~C{\'o}rdoba, R.~Granero-Belinch{\'o}n, and R.~Orive.
\newblock The confined {M}uskat problem: differences with the deep water
  regime.
\newblock {\em Commun. Math. Sci.}, 12(3):423--455, 2014.

\bibitem{e-m10}
J.~Escher and B.~V. Matioc.
\newblock {O}n the parabolicity of the {M}uskat problem: {W}ell-posedness,
  fingering, and stability results.
\newblock {\em Z. Anal. Anwend.}, 30:193--218, 2011.

\bibitem{G}
R.~Granero-Belinch\'on.
\newblock {G}lobal existence for the confined {M}uskat problem.
\newblock {\em Submitted}, 2013.

\bibitem{CXSC}
W.~Hofschuster and W.~Kr{\"a}mer.
\newblock C-{X}{S}{C} 2.0--a {C}++ library for e{X}tended {S}cientific
  {C}omputing.
\newblock In {\em Numerical software with result verification}, pages 15--35.
  Springer, 2004.

\bibitem{Holzmann-Lang-Schutt:gravitation-verified-quadrature}
O.~Holzmann, B.~Lang, and H.~Sch\"utt.
\newblock Newton's constant of gravitation and verified numerical quadrature.
\newblock {\em Reliable Computing}, 2(3):229--239, 1996.

\bibitem{KK}
H.~Kawarada and H.~Koshigoe.
\newblock {U}nsteady flow in porous media with a free surface.
\newblock {\em Japan Journal of Industrial and Applied Mathematics},
  8(1):41--84, 1991.

\bibitem{Kramer-Wedner:adaptive-gauss-legendre-verified-computation}
W.~Kr\"amer and S.~Wedner.
\newblock Two adaptive {G}auss-{L}egendre type algorithms for the verified
  computation of definite integrals.
\newblock {\em Reliable Computing}, 2(3):241--253, 1996.

\bibitem{Lang:multidimensional-verified-gaussian-quadrature}
B.~Lang.
\newblock Derivative-based subdivision in multi-dimensional verified gaussian
  quadrature.
\newblock In G.~Alefeld, J.~Rohn, S.~Rump, and T.~Yamamoto, editors, {\em
  Symbolic Algebraic Methods and Verification Methods}, pages 145--152.
  Springer Vienna, 2001.

\bibitem{Moore-Bierbaum:methods-applications-interval-analysis}
R.~Moore and F.~Bierbaum.
\newblock {\em Methods and applications of interval analysis}, volume~2.
\newblock Society for Industrial \& Applied Mathematics, 1979.

\bibitem{muskat1937flow}
M.~Muskat.
\newblock {\em The Flow of Homogeneous Fluids Through Porous Media}.
\newblock International series in physics. McGraw-Hill Book Company,
  Incorporated, 1937.

\bibitem{Neumaier:branch-and-bound}
A.~Neumaier.
\newblock Complete search in continuous global optimization and constraint
  satisfaction.
\newblock {\em Acta Numer.}, 13:271--369, 2004.

\bibitem{SCH}
M.~Siegel, R.~Caflisch, and S.~Howison.
\newblock {G}lobal existence, singular solutions, and ill-posedness for the
  {M}uskat problem.
\newblock {\em Communications on {P}ure and {A}pplied {M}athematics},
  57(10):1374--1411, 2004.

\end{thebibliography}

\end{document}